\documentclass{siamart}

\usepackage{lipsum}
\usepackage{amsfonts}
\usepackage{graphicx}
\usepackage{epstopdf}
\usepackage{algorithmic}
\ifpdf
  \DeclareGraphicsExtensions{.eps,.pdf,.png,.jpg}
\else
  \DeclareGraphicsExtensions{.eps}
\fi

\title{A Simple Parallel Algorithm with an $O(1/t)$ Convergence Rate for General Convex Programs\thanks{Using the methodology initiated in the current paper, we further developed a different primal-dual type algorithm with the same $O(1/t)$ convergence rate in the extended work \cite{YuNeely16CDC}. }}

\author{
  Hao Yu\thanks{Department of Electrical Engineering, University of Southern California, Los Angeles, CA (\email{yuhao@usc.edu},
    \email{mjneely@usc.edu}).}
  \and
  Michael J. Neely\footnotemark[2]
}

\usepackage{amsopn}

\newtheorem{Def}{Definition}
\newtheorem{Assumption}{Assumption}
\newtheorem{Thm}{Theorem}

\newtheorem{Lem}{Lemma}
\newtheorem{Cor}{Corollary}

\DeclareMathOperator*{\argmin}{argmin}

\begin{document}
\maketitle

\begin{abstract}
This paper considers convex programs with a general (possibly non-differentiable) convex objective function and Lipschitz continuous convex inequality constraint functions. A simple algorithm is developed and achieves an $O(1/t)$ convergence rate. Similar to the classical dual subgradient algorithm and the ADMM algorithm, the new algorithm has a parallel implementation when the objective and constraint functions are separable. However, the new algorithm has a faster $O(1/t)$ convergence rate compared with the best known $O(1/\sqrt{t})$ convergence rate for the dual subgradient algorithm with primal averaging. Further, it can solve convex programs with nonlinear constraints, which cannot be handled by the ADMM algorithm.  The new algorithm is applied to a multipath network utility maximization problem and yields a decentralized flow control algorithm with the fast $O(1/t)$ convergence rate. 
\end{abstract}



\section{Introduction}\label{sec:intro}
Fix positive integers $n$ and $m$. 
Consider the general convex program:
\begin{align}
\text{minimize}  ~\quad&f(\mathbf{x}) \label{eq:program-objective}\\
\text{subject to} \quad&  g_k(\mathbf{x}) \leq  0, \forall k\in\{1,2,\ldots,m\}, \label{eq:program-inequality-constraint}\\
			 &  \mathbf{x}\in \mathcal{X}, \label{eq:program-set-constraint}
\end{align}
where set $\mathcal{X}\subseteq \mathbb{R}^{n}$ is a closed convex set;  function $f(\mathbf{x})$ is continuous and  convex on $\mathcal{X}$; and functions $g_k(\mathbf{x}),\forall k \in\{1,2,\ldots,m\}$ are convex and Lipschitz continuous on $\mathcal{X}$. Note that the functions $f(\mathbf{x}), g_1(\mathbf{x}), \ldots, g_m(\mathbf{x})$ are not necessarily differentiable. Denote the stacked vector of multiple functions $g_1(\mathbf{x}), g_2(\mathbf{x}), \ldots, g_m(\mathbf{x})$ as $\mathbf{g}(\mathbf{x}) = \big[g_1(\mathbf{x}), g_2(\mathbf{x}), \ldots, g_m(\mathbf{x})\big]^T$.  The Lipschitz continuity of each $g_{k}(\mathbf{x})$ implies that $\mathbf{g}(\mathbf{x})$ is Lipschitz continuous on $\mathcal{X}$. Throughout this paper, we use $\Vert \cdot \Vert$ to denote the vector Euclidean norm and Lipschitz continuity is defined with respect to the Euclidean norm. The following assumptions are imposed on the convex program  \eqref{eq:program-objective}-\eqref{eq:program-set-constraint}:

\begin{Assumption}[Basic Assumptions] \label{as:basic}~
\begin{itemize}
\item  There exists a (possibly non-unique) optimal solution $\mathbf{x}^\ast\in \mathcal{X}$ that solves the convex program \eqref{eq:program-objective}-\eqref{eq:program-set-constraint}. 
\item There exists a constant $\beta$ such that $\Vert \mathbf{g}(\mathbf{x}_1) - \mathbf{g}(\mathbf{x}_2)\Vert \leq \beta \Vert \mathbf{x}_1 - \mathbf{x}_2\Vert$ for all $\mathbf{x}_1, \mathbf{x}_2 \in \mathcal{X}$, i.e., the Lipschitz continuous function $\mathbf{g}(\mathbf{x})$ has modulus $\beta$. 
\end{itemize}
\end{Assumption} 

\begin{Assumption}[Existence of Lagrange Multipliers] \label{as:strong-duality} The convex program \eqref{eq:program-objective}-\eqref{eq:program-set-constraint} has Lagrange multipliers attaining the strong duality. That is, there exists a Lagrange multiplier vector $\boldsymbol{\lambda}^\ast = [\lambda_1^\ast, \lambda_2^\ast, \ldots, \lambda_m^\ast]^{T}\geq \mathbf{0}$ such that 
\begin{align*}
q(\boldsymbol{\lambda}^\ast) = \min\limits_{\mathbf{x}\in \mathcal{X}}\Big\{f(\mathbf{x}) : g_k(\mathbf{x})\leq 0, \forall k\in\{1,2,\ldots,m\}\Big\}, 
\end{align*}
where $q(\boldsymbol{\lambda}) = \min\limits_{\mathbf{x}\in \mathcal{X}}\{f(\mathbf{x})+ \sum_{k=1}^m \lambda_k g_k(\mathbf{x})\}
$ is the {\it Lagrangian dual function} of problem \eqref{eq:program-objective}-\eqref{eq:program-set-constraint}.\end{Assumption}
\cref{as:strong-duality} is a mild assumption. For convex programs, \cref{as:strong-duality}  is implied by the existence of a vector $\mathbf{s} \in \mathcal{X}$ such that
$g_k(\mathbf{s}) < 0$ for all $k \in \{1, \ldots, m\}$, called the \emph{Slater  condition}  \cite{book_NonlinearProgrammingTA,book_ConvexOptimization}. However, there are convex programs where  \cref{as:strong-duality}  holds but the Slater condition does not hold. 

\subsection{New Algorithm}

Consider the following algorithm described in \cref{alg:dpp}. The algorithm 
computes vectors $\mathbf{x}(t) \in \mathcal{X}$ for iterations $t \in \{0, 1, 2, \ldots\}$. Define the average over the first $t>0$ iterations as $\overline{\mathbf{x}}(t) = \frac{1}{t}\sum_{\tau=0}^{t-1} \mathbf{x}(\tau)$. The algorithm uses an initial guess vector that is represented as $\mathbf{x}(-1)$ and that is chosen as any vector in $\mathcal{X}$.  The algorithm also uses vector variables $\mathbf{Q}(t) = \big[ Q_1(t), \ldots, Q_m(t)\big]^T$ in the computations.  
 The main result of this paper is that, whenever the parameter $\alpha$ is chosen to satisfy $\alpha \geq \beta^2/2$,  the  vector $\overline{\mathbf{x}}(t)$ closely approximates a solution to the convex program and has an approximation error that decays like $O(1/t)$.  

\begin{algorithm} 
\caption{}
\label{alg:dpp}
Let $\alpha>0$ be a constant parameter. Choose any $\mathbf{x}(-1) \in \mathcal{X}$. Initialize $Q_{k}(0) = \max\{0, -g_{k}(\mathbf{x}(-1))\} , \forall k\in\{1,2,\ldots, m\}$. At each iteration $t\in\{0,1,2,\ldots\}$, observe $\mathbf{x}(t-1)$ and $\mathbf{Q}(t)$ and do the following:
\begin{itemize}
\item  Choose $\mathbf{x}(t)$ as 
\begin{align*}
\mathbf{x}(t)  =\argmin_{\mathbf{x}\in \mathcal{X}} \Big\{ f(\mathbf{x})  + \big[\mathbf{Q}(t) + \mathbf{g}(\mathbf{x}(t-1))\big]^T\mathbf{g}(\mathbf{x}) +  \alpha \Vert \mathbf{x} - \mathbf{x}(t-1)\Vert^{2}\Big\},
\end{align*}

\item Update virtual queues via  \[Q_{k}(t+1) = \max\{-g_{k}(\mathbf{x}(t)), Q_{k}(t) + g_{k}(\mathbf{x}(t))\}, \forall k\in\{1,2,\ldots, m\}.\]
\item Update the averages $\overline{\mathbf{x}}(t)$  via \[\overline{\mathbf{x}}(t+1) = \frac{1}{t+1}\sum_{\tau=0}^{t} \mathbf{x}(\tau) = \overline{\mathbf{x}}(t) \frac{t}{t+1} + \mathbf{x}(t) \frac{1}{t+1}.\]
\end{itemize}
\end{algorithm}

The variables $\mathbf{x}(t)$ are called \emph{primal variables}.  The variables $\mathbf{Q}(t)$ can be viewed as 
\emph{dual variables} because they have a close connection to Lagrange multipliers. The variables $\mathbf{Q}(t)$ shall be called \emph{virtual queues} because their update rule resembles a queueing equation.
The virtual queue update used by \cref{alg:dpp} is related to traditional virtual queue and dual variable update rules.  However, there is an important difference.  A traditional update rule is $Q_k(t+1) = \max\{Q_k(t) + g_k(\mathbf{x}(t)), 0\}$ \cite{book_NonlinearProgramming_Bertsekas,book_Neely10}.  In contrast, the new algorithm takes a max with $-g_k(\mathbf{x}(t))$, rather than a max with $0$.  The primal update rule for $\mathbf{x}(t)$ is also new and is discussed in more detail in the next subsection. 

\cref{alg:dpp} has the following desirable property: 
If the functions $f(\mathbf{x})$ and $\mathbf{g}(\mathbf{x})$ are separable with respect to components or blocks of $\mathbf{x}$, then the primal updates for $\mathbf{x}(t)$ can be decomposed into several smaller independent subproblems, each of which only involves a component or block of $\mathbf{x}(t)$.

\subsection{The Dual Subgradient Algorithm and the Drift-Plus-Penalty Algorithm}

The dual subgradient algorithm is a well known iterative technique that approaches optimality for certain strictly convex problems \cite{book_NonlinearProgramming_Bertsekas}. A modification of the dual subgradient algorithm that averages the resulting sequence of primal estimates is known to solve general convex programs (without requiring strict convexity) and provides an $O(1/\sqrt{t})$ convergence rate\footnote{In \cite{Neely05DCDIS,Nedic09,Neely14Arxiv_ConvergenceTime}, the dual subgradient algorithm with primal averaging is proven to achieve an $\epsilon$-approximate solution with $O(1/\epsilon^2)$ iterations by using an $O(\epsilon)$ step size. We say that the dual subgradient algorithm with primal averaging has an $O(1/\sqrt{t})$ convergence rate because an algorithm with $O(1/\sqrt{t})$ convergence requires the same $O(1/\epsilon^2)$ iterations to yield an $\epsilon$-approximate solution. However, the dual subgradient algorithm in \cite{Neely05DCDIS,Nedic09,Neely14Arxiv_ConvergenceTime} does not have vanishing errors as does an  algorithm with $O(1/\sqrt{t})$ convergence. } \cite{Neely05DCDIS,Nedic09,Neely14Arxiv_ConvergenceTime}, where $t$ is the number of iterations.  Work \cite{Necoara14TAC} improves the convergence rate of the dual subgradient algorithm to $O(1/t)$ in the special case when the objective function $f(\mathbf{x})$ is strongly convex and second order differentiable and constraint functions $g_k(\mathbf{x})$ are second order differentiable and have bounded Jacobians. A recent work in \cite{YuNeely15CDC} shows that the convergence rate of the dual subgradient algorithm with primal averaging is $O(1/t)$ without requiring the differentiability of $f(\mathbf{x})$ and $g_k(\mathbf{x})$ but still requiring the strong convexity of $f(\mathbf{x})$. (Further improvements are also possible under more restrictive assumptions, see \cite{YuNeely15CDC}.) The dual subgradient algorithm with primal averaging is also called the Drift-Plus-Penalty (DPP) algorithm. This is because it is a special case of a stochastic optimization procedure that minimizes a drift expression for a quadratic Lyapunov function \cite{book_Neely10}. One advantage of these dual subgradient and drift approaches is that the computations required at every iteration are simple and can yield parallel algorithms when $f(\mathbf{x})$ and $g_k(\mathbf{x})$ are separable.

\cref{alg:dpp} developed in the current paper maintains this simplicity on every iteration, but provides fast $O(1/t)$ convergence for general convex programs, without requiring strict convexity or strong convexity. For example, the algorithm has the $O(1/t)$ convergence rate in the special case of linear $f(\mathbf{x})$.  \cref{alg:dpp} is similar to a DPP algorithm, or equivalently, a classic dual subgradient algorithm with primal averaging, with the following distinctions: 
\begin{enumerate} 
\item The Lagrange multiplier (``virtual queue") update equation for $Q_k(t)$ is modified to take a max with $-g_k(\mathbf{x}(t))$, rather than simply project onto the nonnegative real numbers.
\item The minimization step augments the $Q_k(t)$ weights with $g_k(\mathbf{x}(t-1))$ values obtained on the previous step. These  $g_k(\mathbf{x}(t-1))$ quantities, when multiplied by constraint functions $g_k(\mathbf{x})$, yield a cross-product term in the primal update. This cross term together with another newly introduced quadratic term in the primal update can cancel a quadratic term in an upper bound of the Lyapunov drift such that a finer analysis of the drift-plus-penalty leads to the fast $O(1/t)$ convergence rate.
\item A quadratic term, which is similar to a term used in proximal algorithms \cite{Parikh13ProximalAlgorithm},  is introduced.  This provides a strong convexity ``pushback". The pushback is not sufficient to alone cancel the main drift components, but it cancels residual components introduced by the new $g_k(\mathbf{x}(t-1))$ weight.   
\end{enumerate}
\subsection{The ADMM Algorithm}The Alternating Direction Method of Multipliers (ADMM) is an algorithm used to solve linear equality constrained convex programs in the following form:
\begin{align}
\text{minimize}~\quad &f_{1}(\mathbf{x}) +  f_{2}(\mathbf{y}) \label{eq:admm-program-obj}\\
\text{subject to} \quad  &  \mathbf{A}\mathbf{x} + \mathbf{B}\mathbf{y} = \mathbf{c}, \label{eq:admm-program-equality-constraint}\\
			 &  \mathbf{x}\in \mathcal{X}, \mathbf{y}\in \mathcal{Y}. \label{eq:admm-program-set-constraint}
\end{align}
Define the augmented Lagrangian as $L_{\rho}(\mathbf{x},\mathbf{y}, \boldsymbol{\lambda}) = f_{1}(\mathbf{x}) + f_{2} (\mathbf{y}) + \boldsymbol{\lambda}^{T}\big(\mathbf{A}\mathbf{x} + \mathbf{B}\mathbf{y}- \mathbf{c}\big) + \frac{\rho}{2} \Vert \mathbf{A}\mathbf{x}+ \mathbf{B}\mathbf{y} -\mathbf{c}\Vert^{2}$. At each iteration $t$, the ADMM algorithm consists of the following steps:
\begin{itemize}
\item Update $\mathbf{x}(t) = \argmin_{\mathbf{x}\in \mathcal{X}} L_{\rho} (\mathbf{x}, \mathbf{y}(t-1),  \boldsymbol{\lambda}(t-1))$.
\item Update $\mathbf{y}(t) = \argmin_{\mathbf{y}\in \mathcal{Y}} L_{\rho} (\mathbf{x}(t), \mathbf{y},  \boldsymbol{\lambda}(t-1))$.
\item Update $\boldsymbol{\lambda}(t) = \boldsymbol{\lambda}(t-1) + \rho \big(\mathbf{A}\mathbf{x}(t)+ \mathbf{B}\mathbf{y}(t) - \mathbf{c}\big)$.
\end{itemize}
Thus, the ADMM algorithm yields a distributed algorithm where the updates of $\mathbf{x}$ and $\mathbf{y}$ only involve local sub-problems and is suitable to solve large scale convex programs in machine learning, network scheduling, computational biology and finance \cite{Boyd11ADMMFoundatationTrends}.

The best known convergence rate of ADMM algorithm for convex program with general convex $f_{1}(\cdot)$ and $f_{2}(\cdot)$ is recently shown to be $O(1/t)$ \cite{He12SIAMNA,Lin15JORSC}.  An asynchronous ADMM algorithm with the same $O(1/t)$ convergence rate is studied in \cite{Wei13AsynchronousADMM}. Note that we can apply \cref{alg:dpp} to solve the problem \eqref{eq:admm-program-obj}-\eqref{eq:admm-program-set-constraint} after replacing the equality constraint $\mathbf{A}\mathbf{x} + \mathbf{B}\mathbf{y} = \mathbf{c}$ by two linear inequality constraints $\mathbf{A}\mathbf{x} + \mathbf{B}\mathbf{y} \leq \mathbf{c}$ and $\mathbf{A}\mathbf{x} + \mathbf{B}\mathbf{y} \geq \mathbf{c}$.\footnote{In fact, we do not need to replace the linear equality constraint with two inequality constraints. We can apply \cref{alg:dpp} to problem \eqref{eq:admm-program-obj}-\eqref{eq:admm-program-set-constraint} directly by modifying the virtual queue update equations as $Q_{k}(t+1) = Q_{k}(t) + g_{k}(\mathbf{x}(t)), \forall k\in\{1,2,\ldots, m\}$. In this case, a simple adaption of the convergence rate analysis in this paper can establish the same $O(1/t)$ convergence rate. However, to simplify the presentation, this paper just considers the general convex program in the form of  \eqref{eq:program-objective}-\eqref{eq:program-set-constraint} since any linear equality can be equivalently represented by two linear inequalities.} It can be observed that the algorithm yielded by \cref{alg:dpp} is also separable for $\mathbf{x}$ and $\mathbf{y}$. In fact, the updates of $\mathbf{x}$ and $\mathbf{y}$ in \cref{alg:dpp} are fully parallel while the ADMM algorithm updates $\mathbf{x}$ and $\mathbf{y}$ sequentially. The remaining part of this paper shows that the convergence rate of \cref{alg:dpp} is also $O(1/t)$.

However, a significant limitation of the ADMM algorithm is that it can only solve problems with linear constraints. In contrast, \cref{alg:dpp} proposed in this paper can solve general convex programs with non-linear  constraints.

\subsection{Decentralized Multipath Network Flow Control Problems}

Section \ref{section:network} presents an example application to multipath network flow control  problems.  The algorithm has a queue-based interpretation that is natural for networks. Prior work on distributed optimization for networks is in \cite{Low99TON,Neely05DCDIS,Beck14,YuNeely15CDC,Wei13TAC-1}. It is known that the DPP algorithm achieves $O(1/\sqrt{t})$ convergence for general networks \cite{Neely05DCDIS}, and several algorithms show faster $O(1/t)$ convergence for special classes of strongly convex problems \cite{Beck14,YuNeely15CDC}. However, multipath network flow problems fundamentally fail to satisfy the strong convexity property because they have routing variables that participate in the constraints but not in the objective function.  The algorithm of the current paper does not require strong convexity. It easily solves this multi-path scenario with fast $O(1/t)$ convergence.  The algorithm also has a simple distributed implementation and allows for general convex but nonlinear constraint functions.

\section{Preliminaries and Basic Analysis}

This section presents useful preliminaries in convex analysis and important facts of \cref{alg:dpp}.

\subsection{Preliminaries}

\begin{Def}[Lipschitz Continuity] \label{def:Lipschitz-continuous}
Let $\mathcal{X} \subseteq \mathbb{R}^n$ be a convex set. Function $h: \mathcal{X}\rightarrow \mathbb{R}^m$ is said to be Lipschitz continuous  on $\mathcal{X}$ with modulus $L$ if there exists $L> 0$ such that $\Vert h(\mathbf{y}) - h(\mathbf{x}) \Vert \leq L \Vert\mathbf{y} - \mathbf{x}\Vert$  for all $ \mathbf{x}, \mathbf{y} \in \mathcal{X}$. 
\end{Def}

\begin{Def}[Strongly Convex Functions]
 Let $\mathcal{X} \subseteq \mathbb{R}^n$ be a convex set. Function $h$ is said to be strongly convex on $\mathcal{X}$ with modulus $\alpha$ if there exists a constant $\alpha>0$ such that $h(\mathbf{x}) - \frac{1}{2} \alpha \Vert \mathbf{x} \Vert^2$ is convex on $\mathcal{X}$.
\end{Def}

By the definition of strongly convex functions, it is easy to show that if $h(\mathbf{x})$ is convex and $\alpha>0$, then $h(\mathbf{x}) + \alpha \Vert \mathbf{x} - \mathbf{x}_0\Vert^2$ is strongly convex with modulus $2\alpha$ for any constant $\mathbf{x}_0$.

\begin{Lem}[Theorem 6.1.2 in \cite{book_FundamentalConvexAnalysis}] \label{lm:strong-convex}
Let $h(\mathbf{x})$ be strongly convex on $\mathcal{X}$ with modulus $\alpha$. Let $\partial h(\mathbf{x})$ be the set of all subgradients of $h$ at point $\mathbf{x}$. Then $h(\mathbf{y}) \geq h(\mathbf{x}) + \mathbf{d}^T (\mathbf{y} - \mathbf{x}) + \frac{\alpha}{2}\Vert \mathbf{y} - \mathbf{x} \Vert^2$ for all $\mathbf{x}, \mathbf{y}\in \mathcal{X}$ and all $\mathbf{d}\in \partial h(\mathbf{x})$.
\end{Lem}

\begin{Lem}[Proposition B.24 (f) in \cite{book_NonlinearProgramming_Bertsekas}] \label{lm:first-order-optimality}
Let $\mathcal{X}\subseteq \mathbb{R}^n$ be a convex set. Let function $h$ be convex on $\mathcal{X}$ and $\mathbf{x}^{opt}$ be a global minimum of $h$ on $\mathcal{X}$.  Let $\partial h(\mathbf{x})$ be the set of all subgradients of $h$ at point $\mathbf{x}$. Then, there exists $\mathbf{d} \in \partial h(\mathbf{x}^{opt})$ such that $\mathbf{d}^T(\mathbf{x} - \mathbf{x}^{opt}) \geq 0$ for all $\mathbf{x}\in \mathcal{X}$. 
\end{Lem}

\begin{Cor} \label{cor:strong-convex-quadratic-optimality}
Let $\mathcal{X} \subseteq \mathbb{R}^{n}$ be a convex set. Let function $h$ be strongly convex on $\mathcal{X}$ with modulus $\alpha$ and $\mathbf{x}^{opt}$ be a global minimum of $h$ on $\mathcal{X}$. Then, $h(\mathbf{x}^{opt}) \leq h(\mathbf{x}) - \frac{\alpha}{2} \Vert \mathbf{x}^{opt} - \mathbf{x}\Vert^{2}$ for all $\mathbf{x}\in \mathcal{X}$.
\end{Cor}
\begin{proof}
A special case when $h$ is differentiable and $\mathcal{X} = \mathbb{R}^{n}$ is Theorem 2.1.8 in \cite{book_ConvexOpt_Nesterov}.  The proof for general $h$ and $\mathcal{X}$ is as follows: Fix $\mathbf{x}\in \mathcal{X}$. By \cref{lm:first-order-optimality}, there exists $\mathbf{d} \in \partial h(\mathbf{x}^{opt})$ such that $\mathbf{d}^T(\mathbf{x} - \mathbf{x}^{opt}) \geq 0$. By \cref{lm:strong-convex}, we also have
\begin{align*}
h(\mathbf{x}) &\geq h(\mathbf{x}^{opt}) +  \mathbf{d}^{T} (\mathbf{x} - \mathbf{x}^{opt}) + \frac{\alpha}{2} \Vert  \mathbf{x} -\mathbf{x}^{opt} \Vert^{2}\\
&\overset{(a)}{\geq} h(\mathbf{x}^{opt}) + \frac{\alpha}{2} \Vert  \mathbf{x} -\mathbf{x}^{opt} \Vert^{2},
\end{align*}
where $(a)$ follows from the fact that $\mathbf{d}^T(\mathbf{x} - \mathbf{x}^{opt}) \geq 0$.
\end{proof}

\subsection{Properties of the Virtual Queues}
\begin{Lem}\label{lm:virtual-queue} In \cref{alg:dpp}, we have
\begin{enumerate}
\item At each iteration $t\in\{0,1,2,\ldots\}$, $Q_k(t)\geq 0$ for all $k\in\{1,2,\ldots,m\}$.
\item At each iteration $t\in\{0,1,2,\ldots\}$, $Q_{k}(t) + g_{k}(\mathbf{x}(t-1))\geq 0$ for all $k\in\{1,2\ldots, m\}$.
\item At iteration $t=0$, $\Vert \mathbf{Q}(0)\Vert^2 \leq \Vert \mathbf{g}(\mathbf{x}(-1))\Vert^2$. At each iteration $t\in\{1,2,\ldots\}$,  $\Vert \mathbf{Q}(t)\Vert^2 \geq \Vert \mathbf{g}(\mathbf{x}(t-1))\Vert^2$.
\end{enumerate}
\end{Lem}

\begin{proof}~
\begin{enumerate}
\item Fix $k\in\{1,2,\ldots, m\}$. Note that $Q_{k}(0)\geq 0$ by the initialization rule $Q_{k}(0) = \max\{0, -g_{k}(\mathbf{x}(-1))\}$. Assume  $Q_k(t) \geq 0$ and consider time $t+1$. If $g_k(\mathbf{x}(t)) \geq 0$, then $Q_k(t+1) = \max\{-g_{k}(\mathbf{x}(t)), Q_{k}(t) + g_{k}(\mathbf{x}(t))\} \geq Q_{k}(t) + g_{k}(\mathbf{x}(t))\geq 0$. If $g_k(\mathbf{x}(t))<0$, then $Q_k(t+1) = \max\{-g_{k}(\mathbf{x}(t)), Q_{k}(t) + g_{k}(\mathbf{x}(t))\}\geq -g_{k}(\mathbf{x}(t)) > 0$.  Thus, $Q_k(t+1) \geq 0$. The result follows by induction.
\item Fix $k\in\{1,2,\ldots, m\}$. Note that $Q_{k}(0) + g_{k}(\mathbf{x}(-1))\geq 0$ by the initialization rule $Q_{k}(0) = \max\{0, -g_{k}(\mathbf{x}(-1))\} \geq -g_{k}(\mathbf{x}(-1))$. For $t\geq 1$, by the virtual queue update equation, we have 
$$Q_{k}(t) = \max\{-g_{k}(\mathbf{x}(t-1)), Q_{k}(t-1) + g_{k}(\mathbf{x}(t-1))\}\geq -g_{k}(\mathbf{x}(t-1)),$$ 
which implies that $Q_{k}(t) + g_{k}(\mathbf{x}(t-1)) \geq 0$.
\item 
\begin{itemize}
\item For $t=0$. Fix $k\in\{1,2,\ldots, m\}$.  Consider the cases $g_k(\mathbf{x}(-1))\geq 0$ and $g_k(\mathbf{x}(-1))<0$ separately. 
If $g_k(\mathbf{x}(-1)) \geq 0$, then $ Q_k(0) = \max\{0,-g_{k}(\mathbf{x}(-1))\} =0$ and so $|Q_k(0)| \leq |g_k(\mathbf{x}(-1))|$. If $g_k(\mathbf{x}(-1)) < 0$, then $Q_k(0) = \max\{0, -g_{k}(\mathbf{x}(-1))\} = -g_k(\mathbf{x}(-1))$. Thus, in both cases, we have $\vert Q_k(0) \vert \leq \vert g_k(\mathbf{x}(-1))\vert$.  Squaring both sides and summing over $k\in\{1,2,\ldots,m\}$ yields $\Vert \mathbf{Q}(0)\Vert^2 \leq \Vert \mathbf{g}(\mathbf{x}(-1))\Vert^2$.  
\item For $t\geq 1$.  Fix $k\in\{1,2,\ldots, m\}$. Consider the cases $g_k(\mathbf{x}(t-1))\geq 0$ and $g_k(\mathbf{x}(t-1))<0$ separately.  If $g_k(\mathbf{x}(t-1)) \geq 0$, then 
\begin{align*}
Q_k(t) &= \max\{-g_{k}(\mathbf{x}(t-1)), Q_{k}(t-1) + g_{k}(\mathbf{x}(t-1))\} \\
&\geq Q_{k}(t-1) + g_{k}(\mathbf{x}(t-1)) \\
&\overset{(a)}{\geq} g_k(\mathbf{x}(t-1))\\
&= |g_k(\mathbf{x}(t-1))|,
\end{align*}
where (a) follows from part 1.  If $g_k(\mathbf{x}(t-1)) < 0$,  then 
\begin{align*}
Q_k(t) &= \max\{-g_{k}(\mathbf{x}(t-1)), Q_{k}(t-1) + g_{k}(\mathbf{x}(t-1))\} \\
&\geq -g_k(\mathbf{x}(t-1))\\
&= |g_k(\mathbf{x}(t-1))|.
\end{align*}
Thus, in both cases, we have $\vert Q_k(t) \vert \geq \vert g_k(\mathbf{x}(t-1))\vert$.  Squaring both sides and summing over $k\in\{1,2,\ldots,m\}$ yields $\Vert \mathbf{Q}(t)\Vert^2 \geq \Vert \mathbf{g}(\mathbf{x}(t-1))\Vert^2$.
\end{itemize}
\end{enumerate}
\end{proof}

\subsection{Properties of the Drift}

Recall that $\mathbf{Q}(t) = \big[ Q_1(t), \ldots, Q_m(t)\big]^T$ is the vector of virtual queue backlogs.  Define  $L(t) = \frac{1}{2} \Vert \mathbf{Q}(t)\Vert^2$. The function $L(t)$ shall be called a \emph{Lyapunov function}. Define the {Lyapunov drift} as 
\begin{align}
\Delta (t) = L(t+1) - L(t) = \frac{1}{2} \big[ \Vert \mathbf{Q}(t+1)\Vert^{2} - \Vert \mathbf{Q}(t)\Vert^{2}\big]. \label{eq:def-drift}
\end{align}

\begin{Lem}\label{lm:drift}At each iteration $t\in\{0,1,2,\ldots\}$ in \cref{alg:dpp}, an upper bound of the Lyapunov drift is given by
\begin{align}
\Delta(t) \leq \mathbf{Q}^T(t) \mathbf{g}(\mathbf{x}(t))  +\Vert \mathbf{g}(\mathbf{x}(t))\Vert^2.  \label{eq:drift}
\end{align}
\end{Lem}

\begin{proof}~
The virtual queue update equations $Q_k(t+1) = \max\{-g_k(\mathbf{x}(t)), Q_k(t) + g_k(\mathbf{x}(t))\}, \forall k\in \{1,2,\ldots,m\}$ can be rewritten as
\begin{align}
Q_k(t+1) = Q_k(t) + \tilde{g}_k(\mathbf{x}(t)), \forall k\in \{1,2,\ldots,m\}, \label{eq:modified-virtual-queue}
\end{align} 
where 
\begin{equation*}
\tilde{g}_k(\mathbf{x}(t)) = \left \{ \begin{array}{cl}  g_k(\mathbf{x}(t)), & \text{if}~Q_k(t) + g_k(\mathbf{x}(t)) \geq -g_k(\mathbf{x}(t))\\
-Q_k(t) - g_k(\mathbf{x}(t)),  & \text{else} \end{array} \right. \forall k.
\end{equation*}

Fix $k\in\{1,2,\ldots, m\}$. Squaring both sides of \eqref{eq:modified-virtual-queue} and dividing by $2$ yields: 
\begin{align*}
&\frac{1}{2}(Q_k (t+1) )^2 \\
= &\frac{1}{2}(Q_k(t))^2 + \frac{1}{2}\big(\tilde{g}_k(\mathbf{x}(t))\big)^2 +  Q_k (t) \tilde{g}_k(\mathbf{x}(t)) \\
= &\frac{1}{2}(Q_k(t))^2 + \frac{1}{2}\big(\tilde{g}_k(\mathbf{x}(t))\big)^2 +  Q_k (t) g_k(\mathbf{x}(t)) +  Q_k (t)\big( \tilde{g}_k(\mathbf{x}(t)) -g_k(\mathbf{x}(t)) \big)  \\
\overset{(a)}{=}& \frac{1}{2}(Q_k(t))^2 + \frac{1}{2}\big(\tilde{g}_k(\mathbf{x}(t))\big)^2 + Q_k (t) g_k(\mathbf{x}(t)) \\& - \big( \tilde{g}_k (\mathbf{x}(t)) + g_k(\mathbf{x}(t))\big)\big( \tilde{g}_k(\mathbf{x}(t)) -g_k(\mathbf{x}(t)) \big)\\
=& \frac{1}{2}(Q_k(t))^2 - \frac{1}{2}\big(\tilde{g}_k(\mathbf{x}(t))\big)^2 + Q_k (t) g_k(\mathbf{x}(t))  + \big(g_k(\mathbf{x}(t))\big)^2 \\
\leq &\frac{1}{2}(Q_k(t))^2 + Q_k (t) g_k(\mathbf{x}(t))  + \big(g_k(\mathbf{x}(t))\big)^2,
\end{align*}
where $(a)$ follows from the fact that $Q_k (t) \big( \tilde{g}_k (\mathbf{x}(t)) -g_k(\mathbf{x}(t)) \big) = -\big(\tilde{g}_k (\mathbf{x}(t)) +g_k(\mathbf{x}(t)) \big)\big( \tilde{g}_k (\mathbf{x}(t)) -g_k(\mathbf{x}(t)) \big)$, which can be shown by considering $\tilde{g}_k (\mathbf{x}(t)) = g_k(\mathbf{x}(t))$ and $\tilde{g}_k (\mathbf{x}(t)) \neq g_k(\mathbf{x}(t))$.
Summing over $k\in\{1,2,\ldots,m\}$ yields 
\begin{align*}
\frac{1}{2}\Vert\mathbf{Q}(t+1)\Vert^{2} \leq \frac{1}{2}\Vert\mathbf{Q}(t)\Vert^{2}  + \mathbf{Q}^T(t ) \mathbf{g}(\mathbf{x}(t)) + \Vert \mathbf{g}(\mathbf{x}(t))\Vert^2.
\end{align*}
Rearranging the terms yields the desired result.
\end{proof}

\section{Convergence Rate Analysis of \cref{alg:dpp} }

This section analyzes the convergence rate of \cref{alg:dpp} for the problem \eqref{eq:program-objective}-\eqref{eq:program-set-constraint}.

\subsection{An Upper Bound of the Drift-Plus-Penalty Expression}

\begin{Lem}\label{lm:dpp-bound}
Let $\mathbf{x}^{\ast}$ be an optimal solution of the problem \eqref{eq:program-objective}-\eqref{eq:program-set-constraint}.  If $\alpha \geq  \frac{1}{2}\beta^{2}$ in \cref{alg:dpp}, then for all $t\geq 0$, we have
\begin{align*}
&\Delta(t) + f(\mathbf{x}(t))  \\
\leq & f(\mathbf{x}^{\ast}) + \alpha \big[ \Vert \mathbf{x} ^{\ast}- \mathbf{x}(t-1)\Vert^{2} -\Vert \mathbf{x}^{\ast} - \mathbf{x}(t)\Vert^{2}\big] +\frac{1}{2} \big[ \Vert \mathbf{g}(\mathbf{x}(t))\Vert^{2} -  \Vert \mathbf{g}(\mathbf{x}(t-1))\Vert^{2} \big],
\end{align*}
where $\beta$ is defined in \cref{as:basic}.
\end{Lem}

\begin{proof} Fix $t\geq 0$.  Note that \cref{lm:virtual-queue} implies that $\mathbf{Q}(t) + \mathbf{g}(\mathbf{x}(t-1))$ is component-wise nonnegative.  Hence, the function $f(\mathbf{x})  + \big[\mathbf{Q}(t) + \mathbf{g}(\mathbf{x}(t-1))\big]^T\mathbf{g}(\mathbf{x})$ is convex with respect to $\mathbf{x}$ on $\mathcal{X}$. Since $ \alpha \Vert \mathbf{x} - \mathbf{x}(t-1)\Vert^{2}$ is strongly convex with respect to $\mathbf{x}$ with modulus $2\alpha$, it follows that
\begin{align*}
f(\mathbf{x})  + \big[\mathbf{Q}(t) + \mathbf{g}(\mathbf{x}(t-1))\big]^T\mathbf{g}(\mathbf{x}) +  \alpha \Vert \mathbf{x} - \mathbf{x}(t-1)\Vert^{2}
\end{align*}
 is strongly convex with respect to $\mathbf{x}$ with modulus $2\alpha$.  

Since $\mathbf{x}(t)$ is chosen to minimize the above strongly convex function, by \cref{cor:strong-convex-quadratic-optimality},  we have
\begin{align}
 &f(\mathbf{x}(t)) + \big[\mathbf{Q}(t) + \mathbf{g}(\mathbf{x}(t-1))\big]^T\mathbf{g}(\mathbf{x}(t))+ \alpha \Vert \mathbf{x}(t) - \mathbf{x}(t-1)\Vert^{2} \nonumber \\
 \leq & f(\mathbf{x}^{\ast}) + \underbrace{\big[\mathbf{Q}(t) + \mathbf{g}(\mathbf{x}(t-1))\big]^T\mathbf{g}(\mathbf{x}^\ast)}_{\leq0}+ \alpha \Vert \mathbf{x} ^{\ast}- \mathbf{x}(t-1)\Vert^{2} - \alpha \Vert \mathbf{x}^{\ast} - \mathbf{x}(t)\Vert^{2} \nonumber \\
\overset{(a)}{\leq} &f(\mathbf{x}^{\ast}) + \alpha \Vert \mathbf{x} ^{\ast}- \mathbf{x}(t-1)\Vert^{2} - \alpha \Vert \mathbf{x}^{\ast} - \mathbf{x}(t)\Vert^{2},  \label{eq:pf-dpp-bound-eq1}
\end{align}
where (a) follows by using the fact that $g_{k}(\mathbf{x}^{\ast})\leq 0$ for all $k\in\{1,2,\ldots,m\}$ and 
$Q_{k}(t) + g_{k}(\mathbf{x}(t-1))\geq 0$ (i.e., part 2 in \cref{lm:virtual-queue}) to eliminate the term marked by an underbrace.

Note that $\mathbf{u}_{1}^{T} \mathbf{u}_{2} = \frac{1}{2} \big[\Vert  \mathbf{u}_{1}\Vert^{2}  + \Vert \mathbf{u}_{2}\Vert^{2} - \Vert \mathbf{u}_{1} - \mathbf{u}_{2}\Vert^{2} \big]$ for any $\mathbf{u}_{1}, \mathbf{u}_{2}\in \mathbb{R}^{m}$. Thus, we have 
\begin{align}\label{eq:pf-dpp-bound-eq2}
\mathbf{g}(\mathbf{x}(t-1))^T \mathbf{g}(\mathbf{x}(t))  =  \frac{1}{2} \big[ \Vert \mathbf{g}(\mathbf{x}(t-1))\Vert^{2}  + \Vert \mathbf{g}(\mathbf{x}(t))\Vert^{2} - \Vert \mathbf{g}(\mathbf{x}(t-1))-\mathbf{g}(\mathbf{x}(t))\Vert^{2} \big]. 
\end{align}

Substituting \eqref{eq:pf-dpp-bound-eq2} into \eqref{eq:pf-dpp-bound-eq1} and rearranging terms  yields
\begin{align*}
 &f(\mathbf{x}(t)) + \mathbf{Q}^{T}(t)\mathbf{g}(\mathbf{x}(t)) \\
 \leq & f(\mathbf{x}^{\ast}) + \alpha \Vert \mathbf{x} ^{\ast}- \mathbf{x}(t-1)\Vert^{2} - \alpha \Vert \mathbf{x}^{\ast} - \mathbf{x}(t)\Vert^{2}  - \alpha \Vert \mathbf{x}(t) - \mathbf{x}(t-1)\Vert^{2}  \\ &+ \frac{1}{2}  \Vert \mathbf{g}(\mathbf{x}(t-1))-\mathbf{g}(\mathbf{x}(t))\Vert^{2} - \frac{1}{2}  \Vert \mathbf{g}(\mathbf{x}(t-1))\Vert^{2} - \frac{1}{2}  \Vert \mathbf{g}(\mathbf{x}(t))\Vert^{2}\\
 \overset{(a)}{\leq} &f(\mathbf{x}^{\ast}) + \alpha \Vert \mathbf{x} ^{\ast}- \mathbf{x}(t-1)\Vert^{2} - \alpha \Vert \mathbf{x}^{\ast} - \mathbf{x}(t)\Vert^{2}  + (\frac{1}{2}\beta^{2}- \alpha) \Vert \mathbf{x}(t) - \mathbf{x}(t-1)\Vert^{2} \\ & - \frac{1}{2}  \Vert \mathbf{g}(\mathbf{x}(t-1))\Vert^{2} - \frac{1}{2}  \Vert \mathbf{g}(\mathbf{x}(t))\Vert^{2} \\
  \overset{(b)}{\leq} &f(\mathbf{x}^{\ast})  + \alpha \Vert \mathbf{x} ^{\ast}- \mathbf{x}(t-1)\Vert^{2} - \alpha \Vert \mathbf{x}^{\ast} - \mathbf{x}(t)\Vert^{2} 
  - \frac{1}{2}  \Vert \mathbf{g}(\mathbf{x}(t-1)) \Vert^{2}  - \frac{1}{2}  \Vert \mathbf{g}(\mathbf{x}(t))\Vert^{2},
 \end{align*}
where (a) follows from the fact that $\Vert \mathbf{g}(\mathbf{x}(t-1)) - \mathbf{g}(\mathbf{x}(t))\Vert \leq \beta \Vert \mathbf{x}(t) - \mathbf{x}(t-1)\Vert$, which further follows from the assumption that $\mathbf{g}(\mathbf{x})$ is Lipschitz continuous with modulus $\beta$; and (b) follows from the fact $\alpha \geq  \frac{1}{2}\beta^{2}$.

Summing \eqref{eq:drift} with the above inequality yields
\begin{align*}
&\Delta(t) + f(\mathbf{x}(t)) \\
\leq &f(\mathbf{x}^{\ast}) + \alpha \big[ \Vert \mathbf{x} ^{\ast}- \mathbf{x}(t-1)\Vert^{2} -\Vert \mathbf{x}^{\ast} - \mathbf{x}(t)\Vert^{2}\big] +\frac{1}{2} \big[ \Vert \mathbf{g}(\mathbf{x}(t))\Vert^{2} -  \Vert \mathbf{g}(\mathbf{x}(t-1))\Vert^{2} \big].
\end{align*}

\end{proof}

\subsection{Objective Value Violations}

\begin{Lem}\label{lm:obj-diff-bound-from-dpp-bound}
Let $\mathbf{x}^{\ast}$ be an optimal solution of the problem \eqref{eq:program-objective}-\eqref{eq:program-set-constraint} and $\beta$ be defined in \cref{as:basic}.  
\begin{enumerate}
\item If $\alpha \geq \frac{1}{2}\beta^{2}$ in \cref{alg:dpp}, then for all $t\geq 1$, we have $\sum_{\tau=0}^{t-1}f(\mathbf{x}(\tau)) \leq t  f(\mathbf{x}^{\ast})   + \alpha \Vert \mathbf{x}^{\ast} - \mathbf{x}(-1)\Vert^{2}$.
\item If $\alpha > \frac{1}{2}\beta^{2}$ in \cref{alg:dpp}, then for all $t\geq 1$, we have $\sum_{\tau=0}^{t-1}f(\mathbf{x}(\tau)) \leq t  f(\mathbf{x}^{\ast})   + \alpha \Vert \mathbf{x}^{\ast} - \mathbf{x}(-1)\Vert^{2} +\frac{\alpha}{2\alpha - \beta^2} \Vert \mathbf{g}(\mathbf{x}^\ast)\Vert^2 -  \frac{1}{2}\Vert \mathbf{Q}(t)\Vert^{2}$.
\end{enumerate}
\end{Lem}
\begin{proof}
By \cref{lm:dpp-bound}, we have $\Delta(\tau) + f(\mathbf{x}(\tau)) \leq f(\mathbf{x}^{\ast}) + \alpha [\Vert \mathbf{x}^{\ast} - \mathbf{x}(\tau-1)\Vert^{2} - \Vert \mathbf{x}^{\ast} - \mathbf{x}(\tau) \Vert ] + \frac{1}{2} [ \Vert \mathbf{g}(\mathbf{x}(\tau))\Vert^{2} -  \Vert \mathbf{g}(\mathbf{x}(\tau-1))\Vert^{2} ]$ for all $\tau\in\{0,1,2,\ldots \}$. Summing over $\tau\in\{0,1,\ldots,t-1\}$ yields
\begin{align*}
&\sum_{\tau=0}^{t-1} \Delta(\tau) + \sum_{\tau=0}^{t-1} f(\mathbf{x}(\tau)) \leq t f(\mathbf{x}^{\ast}) + \alpha \sum_{\tau=0}^{t-1}[ \Vert \mathbf{x}^{\ast} - \mathbf{x}(\tau-1)\Vert^{2} -  \Vert \mathbf{x}^{\ast} - \mathbf{x}(\tau)\Vert^{2}]  \\& \qquad \qquad \qquad \qquad  \qquad \quad  + \frac{1}{2} \sum_{\tau=0}^{t-1} [ \Vert \mathbf{g}(\mathbf{x}(\tau))\Vert^{2} -  \Vert \mathbf{g}(\mathbf{x}(\tau-1))\Vert^{2} ].
\end{align*}
Recalling that $\Delta(\tau) = L(\tau+1) - L(\tau)$ and simplifying summations yields
\begin{align*}
&L(t) - L(0) + \sum_{\tau=0}^{t-1} f(\mathbf{x}(\tau)) \\
\leq &t f(\mathbf{x}^{\ast}) + \alpha \Vert \mathbf{x}^{\ast} - \mathbf{x}(-1)\Vert^{2} -  \alpha\Vert \mathbf{x}^{\ast} - \mathbf{x}(t-1)\Vert^{2} + \frac{1}{2}\Vert \mathbf{g}(\mathbf{x}(t-1))\Vert^{2} - \frac{1}{2} \Vert \mathbf{g}(\mathbf{x}(-1))\Vert^{2}.
\end{align*}
Rearranging terms; and substituting $L(0) = \frac{1}{2} \Vert \mathbf{Q}(0)\Vert^{2} $ and $L(t) = \frac{1}{2} \Vert \mathbf{Q}(t)\Vert^{2}$ yields
\begin{align}
&\sum_{\tau=0}^{t-1} f(\mathbf{x}(\tau)) \nonumber \\
\leq & t f(\mathbf{x}^{\ast}) + \alpha \Vert \mathbf{x}^{\ast} - \mathbf{x}(-1)\Vert^{2}  -  \alpha\Vert \mathbf{x}^{\ast} - \mathbf{x}(t-1)\Vert^{2} + \frac{1}{2}\Vert \mathbf{g}(\mathbf{x}(t-1))\Vert^{2}\nonumber \\ & - \frac{1}{2}\Vert \mathbf{g}(\mathbf{x}(-1))\Vert^{2} +\frac{1}{2} \Vert \mathbf{Q}(0)\Vert^2 - \frac{1}{2} \Vert \mathbf{Q}(t)\Vert^2 \nonumber \\ \overset{(a)}{\leq}&t f(\mathbf{x}^{\ast}) + \alpha \Vert \mathbf{x}^{\ast} - \mathbf{x}(-1)\Vert^{2}  -  \alpha\Vert \mathbf{x}^{\ast} - \mathbf{x}(t-1)\Vert^{2} + \frac{1}{2}\Vert \mathbf{g}(\mathbf{x}(t-1))\Vert^{2}  - \frac{1}{2} \Vert \mathbf{Q}(t)\Vert^2, \label{eq:pf-obj-diff-bound-eq1}
\end{align}
where (a) follows from the fact that $\Vert \mathbf{Q}(0)\Vert \leq \Vert \mathbf{g}(\mathbf{x}(-1))\Vert$, i.e., part 3 in \cref{lm:virtual-queue}.

Next, we present the proof of both parts:
\begin{enumerate}
\item This part follows from the observation that the equation \eqref{eq:pf-obj-diff-bound-eq1} can be further simplified as 
\begin{align*}
\sum_{\tau=0}^{t-1} f(\mathbf{x}(\tau))\overset{(a)}{\leq} & t f(\mathbf{x}^{\ast}) + \alpha \Vert \mathbf{x}^{\ast} - \mathbf{x}(-1)\Vert^{2}  + \frac{1}{2}\Vert \mathbf{g}(\mathbf{x}(t-1))\Vert^{2}- \frac{1}{2} \Vert \mathbf{Q}(t)\Vert^2 \\
\overset{(b)}{\leq}& t f(\mathbf{x}^{\ast}) + \alpha \Vert \mathbf{x}^{\ast} - \mathbf{x}(-1)\Vert^{2},  
\end{align*}
where (a) follows by ignoring the non-positive term $-\alpha \Vert\mathbf{x}^\ast -\mathbf{x}(t-1)\Vert^2$ on the right side and (b) follows from the fact that $\Vert \mathbf{Q}(t)\Vert \geq \Vert \mathbf{g}(\mathbf{x}(t-1))\Vert$, i.e., part 3 in \cref{lm:virtual-queue}.
\item  This part follows by rewriting the equation \eqref{eq:pf-obj-diff-bound-eq1} as 
\begin{align*}
&\sum_{\tau=0}^{t-1} f(\mathbf{x}(\tau))\\
\leq &t f(\mathbf{x}^{\ast}) + \alpha \Vert \mathbf{x}^{\ast} - \mathbf{x}(-1)\Vert^{2} -  \alpha\Vert \mathbf{x}^{\ast} - \mathbf{x}(t-1)\Vert^{2} + \frac{1}{2}\Vert \mathbf{g}(\mathbf{x}(t-1)) - \mathbf{g}(\mathbf{x}^\ast) + \mathbf{g}(\mathbf{x}^\ast)\Vert^{2} \\ &- \frac{1}{2} \Vert \mathbf{Q}(t)\Vert^{2}\\
=&t f(\mathbf{x}^{\ast}) + \alpha \Vert \mathbf{x}^{\ast} - \mathbf{x}(-1)\Vert^{2} -  \alpha\Vert \mathbf{x}^{\ast} - \mathbf{x}(t-1)\Vert^{2} + \frac{1}{2}\Vert \mathbf{g}(\mathbf{x}(t-1)) - \mathbf{g}(\mathbf{x}^\ast)\Vert^{2} \\& + \mathbf{g}^T(\mathbf{x}^\ast)[\mathbf{g}(\mathbf{x}(t-1)) - \mathbf{g}(\mathbf{x}^\ast)]  + \frac{1}{2} \Vert \mathbf{g}(\mathbf{x}^\ast)\Vert^2- \frac{1}{2} \Vert \mathbf{Q}(t)\Vert^{2}\\
\overset{(a)}{\leq} &t f(\mathbf{x}^{\ast}) + \alpha \Vert \mathbf{x}^{\ast} - \mathbf{x}(-1)\Vert^{2} -  \alpha\Vert \mathbf{x}^{\ast} - \mathbf{x}(t-1)\Vert^{2} + \frac{1}{2}\Vert \mathbf{g}(\mathbf{x}(t-1)) - \mathbf{g}(\mathbf{x}^\ast)\Vert^{2} \\& + \Vert \mathbf{g}(\mathbf{x}^\ast) \Vert \Vert \mathbf{g}(\mathbf{x}(t-1)) - \mathbf{g}(\mathbf{x}^\ast)\Vert  + \frac{1}{2} \Vert \mathbf{g}(\mathbf{x}^\ast)\Vert^2- \frac{1}{2} \Vert \mathbf{Q}(t)\Vert^{2}\\
\overset{(b)}{\leq} &t f(\mathbf{x}^{\ast}) + \alpha \Vert \mathbf{x}^{\ast} - \mathbf{x}(-1)\Vert^{2} -  \alpha\Vert \mathbf{x}^{\ast} - \mathbf{x}(t-1)\Vert^{2} + \frac{1}{2} \beta^2 \Vert\mathbf{x}^\ast - \mathbf{x}(t-1)\Vert^{2} \\& + \beta \Vert \mathbf{g}(\mathbf{x}^\ast) \Vert \Vert \mathbf{x}^\ast - \mathbf{x}(t-1)\Vert  + \frac{1}{2} \Vert \mathbf{g}(\mathbf{x}^\ast)\Vert^2- \frac{1}{2} \Vert \mathbf{Q}(t)\Vert^{2}\\
=& t f(\mathbf{x}^{\ast}) + \alpha \Vert \mathbf{x}^{\ast} - \mathbf{x}(-1)\Vert^{2} -  \big(\alpha- \frac{1}{2}\beta^{2}\big)\Big[ \Vert \mathbf{x}^{\ast} - \mathbf{x}(t-1)\Vert - \frac{1}{2}\frac{\beta}{\alpha -\frac{1}{2}\beta^{2}} \Vert \mathbf{g}(\mathbf{x}^\ast)\Vert \Big]^2 \\& + \frac{\alpha}{2\alpha - \beta^2} \Vert \mathbf{g}(\mathbf{x}^\ast)\Vert^2- \frac{1}{2} \Vert \mathbf{Q}(t)\Vert^{2}
\end{align*}
\begin{align*}
 \overset{(c)}{\leq} &t f(\mathbf{x}^{\ast}) + \alpha \Vert \mathbf{x}^{\ast} - \mathbf{x}(-1)\Vert^{2} + \frac{\alpha}{2\alpha - \beta^2} \Vert \mathbf{g}(\mathbf{x}^\ast)\Vert^2- \frac{1}{2} \Vert \mathbf{Q}(t)\Vert^{2},
\end{align*}
where (a) follows from Cauchy-Schwarz inequality; (b) follows from the fact that $\Vert \mathbf{g}(\mathbf{x}(t-1)) - \mathbf{g}(\mathbf{x}^\ast) \Vert \leq \beta \Vert  \mathbf{x}^\ast - \mathbf{x}(t-1) \Vert$, which further follows from the assumption that $\mathbf{g}(\mathbf{x})$ is Lipschitz continuous with modulus $\beta$; and (c) follows from the fact that $\alpha > \frac{1}{2}\beta^{2}$.
\end{enumerate}
\end{proof}

\begin{Thm}[Objective Value Violations]
Let $\mathbf{x}^{\ast}$ be an optimal solution of the problem \eqref{eq:program-objective}-\eqref{eq:program-set-constraint}.  If $\alpha \geq  \frac{1}{2}\beta^{2}$ in \cref{alg:dpp}, for all $t\geq1$, we have 
\begin{align*}
f(\overline{\mathbf{x}}(t))  \leq  f(\mathbf{x}^{\ast}) +\frac{\alpha}{t} \Vert \mathbf{x}^{\ast} - \mathbf{x}(-1)\Vert^{2},
\end{align*}
where $\beta$ is defined in \cref{as:basic}.
\end{Thm}
\begin{proof}
Fix $t\geq$1. By part 1 in \cref{lm:obj-diff-bound-from-dpp-bound}, we have
\begin{align*}
&\sum_{\tau=0}^{t-1}f(\mathbf{x}(\tau)) \leq t f(\mathbf{x}^{\ast}) + \alpha \Vert \mathbf{x}^{\ast} - \mathbf{x}(-1)\Vert^{2}\\
\Rightarrow& \frac{1}{t}\sum_{\tau=0}^{t-1}f(\mathbf{x}(\tau)) \leq f(\mathbf{x}^{\ast}) +\frac{\alpha}{t} \Vert \mathbf{x}^{\ast} - \mathbf{x}(-1)\Vert^{2}.
\end{align*}

Since $\overline{\mathbf{x}}(t) = \frac{1}{t}\sum_{\tau=0}^{t-1} \mathbf{x}(\tau)$ and $f(\mathbf{x})$ is convex, by Jensen's inequality it follows that
\begin{align*}
f(\overline{\mathbf{x}}(t)) \leq \frac{1}{t}\sum_{\tau=0}^{t-1} f(\mathbf{x}(\tau)). 
\end{align*}
\end{proof}

The above theorem shows that the error gap between $f(\overline{\mathbf{x}}(t))$ and the optimal value $f(\mathbf{x}^*)$ is at most $O(1/t)$.  This holds for any initial guess vector $\mathbf{x}(-1) \in \mathcal{X}$.  Of course, choosing $\mathbf{x}(-1)$ close to $\mathbf{x}^*$ is desirable because it reduces 
the coefficient $\alpha\Vert \mathbf{x}^*- \mathbf{x}(-1) \Vert^2$.  

\subsection{Constraint Violations}

\begin{Lem}\label{lm:queue-constraint-inequality}
Let $\mathbf{Q}(t), t\in\{0,1,\ldots\}$ be the sequence generated by \cref{alg:dpp}.  
For any $t\geq 1$, 
\[ Q_k(t) \geq   \displaystyle{\sum_{\tau=0}^{t-1} g_k(\mathbf{x}(\tau))}, \forall k\in\{1,2,\ldots,m\}. \]
\end{Lem}
\begin{proof}
Fix $k\in\{1,2,\ldots,m\}$ and $t \geq 1$.  For any $\tau \in \{0, \ldots, t-1\}$ the update rule of \cref{alg:dpp} gives: 
\begin{align*}
Q_k(\tau+1) &= \max\{-g_{k}(\mathbf{x}(\tau)), Q_k(\tau)+g_k(\mathbf{x}(\tau))\} \\
&\geq Q_k(\tau) + g_k(\mathbf{x}(\tau)). 
\end{align*}
Hence, $Q_k(\tau+1) - Q_k(\tau) \geq g_k(\mathbf{x}(\tau))$.  
Summing over $\tau \in \{0, \ldots, t-1\}$ and using $Q_k(0)\geq 0$ gives the result. 
\end{proof}

\begin{Lem}\label{lm:obj-diff-bound-from-strong-duality}
Let $\mathbf{x}^{\ast}$ be an optimal solution of the problem \eqref{eq:program-objective}-\eqref{eq:program-set-constraint} and $\boldsymbol{\lambda}^\ast$ be a Lagrange multiplier vector satisfying \cref{as:strong-duality}. Let $\mathbf{x}(t), \mathbf{Q}(t), t\in\{0,1,\ldots\}$ be sequences generated by \cref{alg:dpp}. Then,
\begin{align*}
\sum_{\tau=0}^{t-1} f(\mathbf{x}(\tau)) \geq t f(\mathbf{x}^\ast) -  \Vert \boldsymbol{\lambda}^\ast\Vert \Vert \mathbf{Q}(t)\Vert, \quad \forall t\geq 1. 
\end{align*}
\end{Lem}

\begin{proof}
The proof is similar to a related result in \cite{YuNeely15CDC} for the DPP algorithm. Define Lagrangian dual function $q(\boldsymbol{\lambda}) = \min\limits_{\mathbf{x}\in \mathcal{X}}\{f(\mathbf{x})+ \sum_{k=1}^m \lambda_k g_k(\mathbf{x})\}$. For all $\tau\in\{0,1,\ldots\}$, by \cref{as:strong-duality}, we have
\begin{align*}
f(\mathbf{x}^{\ast}) = q(\boldsymbol{\lambda}^{\ast}) \overset{(a)}{\leq}  f(\mathbf{x}(\tau)) + \sum_{k=1}^m \lambda_k^\ast g_k(\mathbf{x}(\tau)),
\end{align*} 
where (a) follows the definition of $q(\boldsymbol{\lambda}^{\ast})$. Thus, we have
\begin{align*}
f(\mathbf{x}(\tau))  \geq f(\mathbf{x}^\ast) -\sum_{k=1}^m \lambda_k^\ast g_k(\mathbf{x}(\tau)),\forall \tau \in\{0,1,\ldots\}.
\end{align*} 

Summing over $\tau\in \{0,1,\ldots, t-1\}$ yields 
\begin{align}
\sum_{\tau=0}^{t-1} f(\mathbf{x}(\tau)) \geq &t f(\mathbf{x}^\ast) -  \sum_{\tau=0}^{t-1} \sum_{k=1}^m\lambda_k^\ast g_k(\mathbf{x}(\tau))\nonumber \\
=&t f(\mathbf{x}^\ast)  - \sum_{k=1}^m  \lambda_k^\ast \Big[\sum_{\tau=0}^{t-1}g_k(\mathbf{x}(\tau))\Big] \nonumber\\
\overset{(a)}{\geq}& t f(\mathbf{x}^\ast)  -\sum_{k=1}^m  \lambda_k^\ast Q_k(t) \nonumber\\
\overset{(b)}{\geq}& t f(\mathbf{x}^\ast)  -\Vert \boldsymbol{\lambda}^\ast\Vert \Vert \mathbf{Q}(t)\Vert, \nonumber
\end{align}
where $(a)$ follows from \cref{lm:queue-constraint-inequality} and the fact that $\lambda_k^\ast \geq 0, \forall k\in\{1,2,\ldots, m\}$; and $(b)$ follows from the Cauchy-Schwarz inequality. \end{proof}

\begin{Lem} \label{lm:queue-bound}
Let $\mathbf{x}^{\ast}$ be an optimal solution of the problem \eqref{eq:program-objective}-\eqref{eq:program-set-constraint} and $\boldsymbol{\lambda}^\ast$ be a Lagrange multiplier vector satisfying \cref{as:strong-duality}. If $\alpha > \frac{\beta^{2}}{2}$ in \cref{alg:dpp}, then for all $t\geq 1$, the virtual queue vector satisfies
\begin{align*}
\Vert \mathbf{Q}(t) \Vert \leq  2 \Vert \boldsymbol{\lambda}^\ast \Vert + \sqrt{ 2\alpha} \Vert \mathbf{x}^{\ast} - \mathbf{x}(-1)\Vert + \sqrt{\frac{\alpha}{\alpha - \frac{1}{2}\beta^{2}}} \Vert \mathbf{g}(\mathbf{x}^\ast)\Vert, 
\end{align*}
where $\beta$ is defined in \cref{as:basic}.
\end{Lem}
\begin{proof}
Fix $t\geq 1$. By part 2 in \cref{lm:obj-diff-bound-from-dpp-bound}, we have
\begin{align*}
\sum_{\tau=0}^{t-1}f(\mathbf{x}(\tau)) \leq & t  f(\mathbf{x}^{\ast}) + \alpha \Vert \mathbf{x}^{\ast} - \mathbf{x}(-1)\Vert^{2} + \frac{\alpha}{2\alpha - \beta^2} \Vert \mathbf{g}(\mathbf{x}^\ast)\Vert^2 -  \frac{1}{2} \Vert \mathbf{Q}(t)\Vert^{2}.
\end{align*}
By \cref{lm:obj-diff-bound-from-strong-duality}, we have
\begin{align*}
\sum_{\tau=0}^{t-1} f(\mathbf{x}(\tau)) \geq tf(\mathbf{x}^\ast) -  \Vert \boldsymbol{\lambda}^\ast\Vert \Vert \mathbf{Q}(t)\Vert.
\end{align*}
Combining the last two inequalities and cancelling the common term $tf(\mathbf{x}^\ast)$ on both sides yields
\begin{align*}
 &\frac{1}{2} \Vert \mathbf{Q}(t)\Vert^{2} - \big(\alpha \Vert \mathbf{x}^{\ast} - \mathbf{x}(-1)\Vert^{2} + \frac{\alpha}{2\alpha - \beta^2} \Vert \mathbf{g}(\mathbf{x}^\ast)\Vert^2\big)\leq \Vert \boldsymbol{\lambda}^\ast\Vert \Vert \mathbf{Q}(t)\Vert \\
\Rightarrow & \big( \Vert \mathbf{Q}(t)\Vert - \Vert \boldsymbol{\lambda}^\ast\Vert \big)^{2} \leq \Vert \boldsymbol{\lambda}^\ast\Vert^{2} + 2\alpha \Vert \mathbf{x}^{\ast} - \mathbf{x}(-1)\Vert^{2} + \frac{\alpha}{\alpha - \frac{1}{2}\beta^{2}} \Vert \mathbf{g}(\mathbf{x}^\ast)\Vert^2\\
\Rightarrow & \Vert \mathbf{Q}(t)\Vert \leq \Vert \boldsymbol{\lambda}^\ast \Vert + \sqrt{\Vert \boldsymbol{\lambda}^\ast\Vert^{2} + 2\alpha \Vert \mathbf{x}^{\ast} - \mathbf{x}(-1)\Vert^{2} +  \frac{\alpha}{\alpha - \frac{1}{2}\beta^{2}} \Vert \mathbf{g}(\mathbf{x}^\ast)\Vert^2}\\
\overset{(a)}{\Rightarrow}& \Vert \mathbf{Q}(t)\Vert \leq 2 \Vert \boldsymbol{\lambda}^\ast \Vert + \sqrt{ 2\alpha} \Vert \mathbf{x}^{\ast} - \mathbf{x}(-1)\Vert + \sqrt{\frac{\alpha}{\alpha - \frac{1}{2}\beta^{2}}} \Vert \mathbf{g}(\mathbf{x}^\ast)\Vert,
\end{align*}
where (a) follows from the basic inequality $\sqrt{a+b+c}\leq \sqrt{a} + \sqrt{b} +\sqrt{c}$ for any $a,b,c\geq0$.
\end{proof}

\begin{Thm}[Constraint Violations]
Let $\mathbf{x}^{\ast}$ be an optimal solution of the problem \eqref{eq:program-objective}-\eqref{eq:program-set-constraint} and $\boldsymbol{\lambda}^\ast$ be a Lagrange multiplier vector satisfying \cref{as:strong-duality}.  If $\alpha > \frac{\beta^{2}}{2}$ in \cref{alg:dpp}, then for all $t\geq1$, the constraint functions satisfy
\begin{align*}
g_{k}(\overline{\mathbf{x}}(t)) \leq  \frac{1}{t} \Big(2 \Vert \boldsymbol{\lambda}^\ast \Vert +\sqrt{ 2\alpha} \Vert \mathbf{x}^{\ast} - \mathbf{x}(-1)\Vert + \sqrt{\frac{\alpha}{\alpha - \frac{1}{2}\beta^{2}}} \Vert \mathbf{g}(\mathbf{x}^\ast)\Vert \Big), \forall k\in\{1,2,\ldots, m\},
\end{align*}
where $\beta$ is defined in \cref{as:basic}.
\end{Thm}
\begin{proof}
Fix $t\geq 1$ and $k\in\{1,2,\ldots,m\}$. Recall that $\overline{\mathbf{x}}(t) = \frac{1}{t}\sum_{\tau=0}^{t-1} \mathbf{x}(\tau)$. Thus, 
\begin{align*}
g_k (\overline{\mathbf{x}}(t)) &\overset{(a)}{\leq} \frac{1}{t} \sum_{\tau=0}^{t-1} g_k(\mathbf{x}(\tau)) \\
 &\overset{(b)}{\leq} \frac{Q_k(t)}{t} \\
 &\leq \frac{\Vert \mathbf{Q}(t)\Vert}{t}\\
 &\overset{(c)}{\leq} \frac{1}{t}\Big(2 \Vert \boldsymbol{\lambda}^\ast \Vert + \sqrt{ 2\alpha} \Vert \mathbf{x}^{\ast} - \mathbf{x}(-1)\Vert + \sqrt{\frac{\alpha}{\alpha - \frac{1}{2}\beta^{2}}} \Vert \mathbf{g}(\mathbf{x}^\ast)\Vert\Big),
 \end{align*}
where (a) follows from the convexity of $g_k(\mathbf{x}), k\in\{1,2,\ldots,m\}$ and Jensen's inequality; (b) follows from \cref{lm:queue-constraint-inequality}; and (c) follows from \cref{lm:queue-bound}.
\end{proof}

\subsection{Convergence Rate of \cref{alg:dpp}}

The next theorem summarizes the last two subsections.
\begin{Thm}\label{thm:overall-convergence}
Let $\mathbf{x}^{\ast}$ be an optimal solution of the problem \eqref{eq:program-objective}-\eqref{eq:program-set-constraint} and $\boldsymbol{\lambda}^\ast$ be a Lagrange multiplier vector satisfying \cref{as:strong-duality}.  If $\alpha > \frac{\beta^{2}}{2}$ in \cref{alg:dpp}, then for all $t\geq 1$, we have
 \begin{align*}
f(\overline{\mathbf{x}}(t))  \leq & f(\mathbf{x}^{\ast}) +\frac{\alpha}{t}\Vert \mathbf{x}^{\ast} - \mathbf{x}(-1)\Vert^{2}, \\
g_{k}(\overline{\mathbf{x}}(t)) \leq&  \frac{1}{t} \Big(2 \Vert \boldsymbol{\lambda}^\ast \Vert +\sqrt{ 2\alpha} \Vert \mathbf{x}^{\ast} - \mathbf{x}(-1)\Vert + \sqrt{\frac{\alpha}{\alpha - \frac{1}{2}\beta^{2}}} \Vert \mathbf{g}(\mathbf{x}^\ast)\Vert \Big), \forall k\in\{1,2,\ldots, m\},
\end{align*}
where $\beta$ is defined in \cref{as:basic}. In summary, \cref{alg:dpp} ensures error decays like $O(1/t)$ and provides an $\epsilon$-approximate solution with convergence time $O(1/\epsilon)$.
\end{Thm}

\section{Application: Decentralized Network Utility Maximization} \label{section:network} 

This section considers the application of \cref{alg:dpp} to decentralized multipath network utility maximization problems.

\subsection{Decentralized Multipath Flow Control}

Network flow control can be formulated as the convex optimization of maximizing network utility subject to link capacity constraints \cite{Kelly98JORS}. In this view, many existing TCP protocols can be interpreted as distributed solutions to network utility maximization (NUM) problems \cite{Low00DualityModelTCP}. 

In single path network flow control problems, if the utility functions are strictly convex, work \cite{Low99TON} shows that the dual subgradient algorithm (with convergence rate $O(1/\sqrt{t})$) can yield a distributed flow control algorithm.  If the utility functions are only convex but not necessarily strictly convex, the DPP algorithm (or dual subgradient algorithm with primal averaging) can yield a distributed flow control algorithm with an $O(1/\sqrt{t})$ convergence rate \cite{Neely05DCDIS,Nedic09,Neely14Arxiv_ConvergenceTime}. If utility functions are strongly convex, a faster network flow control algorithm with an $O(1/t)$ convergence rate is proposed in \cite{Beck14}. A recent work \cite{YuNeely15CDC} shows that the distributed network flow control based on the DPP algorithm also has convergence rate $O(1/t)$ if utility functions are strongly convex. Other Newton method based distributed algorithms for network flow control with strictly convex utility functions are considered in \cite{Wei13TAC-1}.

However, in multipath network flow control problems, even if the utility function is strictly or strongly convex with respect to the source rate, it is no longer strictly or strongly convex with respect to path rates. Thus, many of the above algorithms requiring strict or strong convexity can no longer be applied.  The DPP algorithm can still be applied but the convergence rate is only $O(1/\sqrt{t})$. Distributed algorithms based on the primal-dual subgradient method, also known as the Arrow-Hurwicz-Uzawa subgradient method, have been considered in \cite{Low03PE}. However, the convergence rate\footnote{Similar to the dual subgradient method, the primal-dual subgradient method has an $O(1/\sqrt{t})$ convergence rate for general convex programs in the sense that $O(1/\epsilon^2)$ iterations are required to obtain an $\epsilon$-approximate solution. However, the primal-dual subgradient method does not have vanishing errors.} of the primal-dual subgradient method for general convex programs without strong convexity is known to be $O(1/\sqrt{t})$ \cite{Nedic09_PrimalDualSubgradient}.  

As shown in the previous sections, \cref{alg:dpp} has an $O(1/t)$ convergence rate for general convex programs and yields a distributed algorithm if the objective function and constraint functions are separable.  The next subsection applies \cref{alg:dpp} to the multipath network flow control problem. The resulting algorithm has  structural 
properties and implementation characteristics similar to the subgradient-based algorithm of \cite{Low00DualityModelTCP} and to the DPP algorithm of \cite{YuNeely15CDC}, but has additional decision variables due to the multipath formulation. (The algorithms in \cite{Low00DualityModelTCP, YuNeely15CDC} rely on strict and strong convexity of the objective function, respectively, and 
apply only to single path situations.)
\subsection{Decentralized Multipath Flow Control Based on \cref{alg:dpp}}
Suppose there are $S$ sources enumerated by $\mathcal{S}=\{1,2,\ldots,S\}$  and $L$ links enumerated 
by $\mathcal{L}=\{1,2,\ldots,L\}$.  Each link $l \in \mathcal{L}$ has a link capacity of $c_l$ bits/slot. Each source sends data from a specific origin to a specific destination, and has multiple path options. The paths for each source can use overlapping links, and they can also overlap with paths of other sources.  Further, two distinct sources can have identical paths.  However, it is useful to give distinct path indices to paths associated with each source
(even if their corresponding paths are physically the same). Specifically, for each source $s$, define $\mathcal{P}_s$ as the set of path indices used by source $s$. The index sets $\mathcal{P}_1, \ldots, \mathcal{P}_S$ are disjoint and $\mathcal{P}_1 \cup \mathcal{P}_2 \cup \cdots \cup\mathcal{P}_S = \mathcal{K} = \{1,2,\ldots, K\}$, where $K$ is the number of path indices.

For each source $s$, let $U_s(y_s)$ be a real-valued, concave, continuous and nondecreasing utility function defined for $y_s\geq 0$. This represents the satisfaction source $s$ receives by communicating with a total rate of $y_s$, where the total rate sums over all paths it used.

Note that different paths can share links in common. Define $\mathcal{D}_l\subseteq \mathcal{K}$ as the set of paths that use link $l$ and $\mathcal{E}_{k} \subseteq \mathcal{L}$ as the set of links used by path $k$. Let $\mathbf{x} = [x_1,\ldots, x_K]^T$ be the vector that specifies the flow rate on each path; and  $\mathbf{y} = [y_1, \ldots, y_S]^T$ be the vector that specifies the rate of each source. 

The goal is to allocate flow rates on each path so that no link is overloaded and network utility is maximized.  This multipath network utility maximization problem can be formulated as follows:
\begin{align}
\text{maximize} ~\quad & \sum_{s=1}^{S} U_{s}(y_{s}) \label{eq:dude1}  \\
\text{subject to} \quad  &  \sum_{k\in \mathcal{D}_l} x_k \leq c_l, \forall l\in \mathcal{L}, \label{eq:dude2} \\
			 &  y_{s} = \sum_{k\in \mathcal{P}_s} x_{k}  , \forall s\in \mathcal{S}, \label{eq:dude3} \\
			 &  0\leq x_{k} \leq x_k^{\max}, \forall k\in\mathcal{K}, \label{eq:dude4}  \\
			 & 0 \leq y_{s} \leq y_{s}^{\max}, \forall s\in \mathcal{S}, \label{eq:dude5} 
\end{align}
where $x^{\max}$ and $y^{\max}$ are the allowed maximum path rate and maximum source rate, respectively.  The expression \eqref{eq:dude1} represents the network utility;  inequality \eqref{eq:dude2} specifies the link capacity constraints; and equality 
\eqref{eq:dude3} enforces the definition of $y_s$.   The linear equality constraints \eqref{eq:dude3} can be formally treated by writing each one as two linear inequality constraints.  However, since the utility functions $U_s(\cdot)$ are nondecreasing and seek to maximize the $y_s$ values, it is clear that the above problem is equivalent to the following: 
\begin{align}
\text{minimize}~\quad & -\sum_{s=1}^{S} U_{s}(y_{s}) \label{eq:network-obj}\\
\text{subject to} \quad  &  \sum_{k\in \mathcal{D}_l} x_k \leq c_l, \forall l\in \mathcal{L},\label{eq:network-link-capacity-cons}\\
			 &  y_{s} \leq \sum_{k\in \mathcal{P}_s} x_{k}  , \forall s\in \mathcal{S}, \label{eq:network-source-rate-cons}\\
			 &  0\leq x_{k} \leq x_{k}^{\max}, \forall k\in\mathcal{K}, \label{eq:network-path-box-cons}\\
			 & 0 \leq y_{s} \leq y_{s}^{\max}, \forall s\in \mathcal{S}. \label{eq:network-source-box-cons}
\end{align}

Note that the constraints \eqref{eq:network-link-capacity-cons}-\eqref{eq:network-source-rate-cons} are linear and can be written as $\mathbf{A}\left[\begin{array}{c}\mathbf{x}\\ \mathbf{y}\end{array}\right] \leq \left[\begin{array}{c}\mathbf{c}\\ \mathbf{0}\end{array}\right]$, where $\mathbf{A}$ is an $(L+S)\times (K+S)$ matrix of which each entry is in $\{0,\pm1\}$.  This inequality constraint $\mathbf{A}\left[\begin{array}{c}\mathbf{x}\\ \mathbf{y}\end{array}\right] \leq \left[\begin{array}{c}\mathbf{c}\\ \mathbf{0}\end{array}\right]$ can be treated as the inequality constraint \eqref{eq:program-inequality-constraint} defined as $\mathbf{g}(\mathbf{z}) = \mathbf{A}\mathbf{z} - \mathbf{b} \leq \mathbf{0}$ with $\mathbf{z} = \left[\begin{array}{c}\mathbf{x}\\ \mathbf{y}\end{array}\right]$ and $\mathbf{b}= \left[\begin{array}{c}\mathbf{c}\\ \mathbf{0}\end{array}\right]$ in the general convex program \eqref{eq:program-objective}-\eqref{eq:program-set-constraint}. The box constraints \eqref{eq:network-path-box-cons}-\eqref{eq:network-source-box-cons} can be treated as the set constraint \eqref{eq:program-set-constraint} in the general convex program \eqref{eq:program-objective}-\eqref{eq:program-set-constraint}.  Thus, the multipath network utility maximization problem \eqref{eq:network-obj}-\eqref{eq:network-source-box-cons} is a special case of the general convex program \eqref{eq:program-objective}-\eqref{eq:program-set-constraint}.

The Lipschitz continuity of $\mathbf{g}(\mathbf{z})$ is summarized in the next lemma.
\begin{Lem}\label{lm:network-beta-bound}
If the constraints \eqref{eq:network-link-capacity-cons}-\eqref{eq:network-source-rate-cons} are written as $\mathbf{g}(\mathbf{z}) \leq \mathbf{0}$, then we have 
\begin{enumerate}
\item The function $\mathbf{g}(\mathbf{z})$ is Lipschitz continuous with modulus 
\begin{align*}
\beta \leq  \sqrt{ S+K + \sum_{k\in \mathcal{K}} d_k},
\end{align*}
where $d_k$ is the length, i.e., number of hops, of path $k\in \mathcal{K}$. 
\item The function $\mathbf{g}(\mathbf{z})$ is Lipschitz continuous with modulus  $\beta\leq \sqrt{(L+1)K+S}$.
\end{enumerate}
\end{Lem}
\begin{proof}~
\begin{enumerate}
\item Recall that $\mathbf{g}(\mathbf{z}) = \mathbf{A}\mathbf{z} - \mathbf{b}$ is Lipschitz continuous with modulus $\beta = \sigma_{\max}(\mathbf{A})$, where $\sigma_{\max}(\mathbf{A})$ is the maximum singular value of matrix $\mathbf{A}$, and a simple upper bound of $\sigma_{\max}(\mathbf{A})$ is the Frobenius norm given by $\Vert \mathbf{A}\Vert_F = \text{tr}(\mathbf{A}^T\mathbf{A}) = \sqrt{\sum_{i,j} A_{ij}^2}$.  Note that $\mathbf{A}$  can be written as 
\begin{align*}
\mathbf{A} = \left[ \begin{array}{cc} \mathbf{R} & \mathbf{0}_{L\times S} \\ -\mathbf{T} & \mathbf{I}_{S} \end{array}\right],
\end{align*}
where $\mathbf{R}$ is a $\{0,1\}$ matrix of size $L\times K$ , $\mathbf{0}_{L\times S}$ is an $L\times S$ zero matrix, $\mathbf{T}$ is a $\{0,1\}$ matrix of size $S\times K$ and $\mathbf{I}_{S}$ is an $S\times S$ identity matrix.  Note that the $(l,k)$-th entry of $\mathbf{R}$ is $1$ if and only if path $k$ uses link $l$;  the $(s,k)$-th entry of $\mathbf{T}$ is $1$ if and only if path $k$ is a path for source $s$, i.e, $k\in \mathcal{P}_s$.  Matrix $\mathbf{R}$ has  $\sum_{k\in \mathcal{K}} d_k$ non-zero entries since each column has exactly $d_k$ non-zero entries. Matrix $\mathbf{T}$ has $K$ non-zero entries since there are in total $K$ network paths.  Thus, matrix $\mathbf{A}$ in total has $S+K+\sum_{k\in \mathcal{K}} d_k$ non-zero entries whose absolute values are equal to $1$. It follows that $\beta \leq \sqrt{ S+K + \sum_{k\in \mathcal{K}} d_k}$.
\item  This part follows from the fact that the length of each path is at most $L$.
\end{enumerate}
\end{proof}

Note that the second bound in the above lemma is more loose than the first one but holds regardless of the flow path configurations in the network. A straightforward application of \cref{alg:dpp} yields the following decentralized network flow control algorithm described in \cref{alg:network-flow}. Similar to the flow control based on the dual subgradient algorithm,  \cref{alg:network-flow} is decentralized and can be easily implemented within the current TCP protocols  \cite{Low00DualityModelTCP}.

\begin{algorithm} 
\caption{}
\label{alg:network-flow}
\begin{itemize}
\item Initialization: Let $x_{k}(-1)\in [0, x_{k}^{\max}], \forall k$ be arbitrary, $y_{s}(-1)\in [0, y_{s}^{\max}], \forall s$ be arbitrary, $Q_{l}(0) = \max\{0, - \sum_{k\in \mathcal{D}_{l}} x_{k}(-1) + c_{l}\}, \forall l\in \mathcal{L}$, $Y_l(0) = Q_l(0) + \sum_{k\in \mathcal{D}_{l}} x_{k}(-1) - c_{l}, \forall l\in \mathcal{L}$, $R_{s}(0) = \max\{0, \sum_{k\in \mathcal{P}_{s}} x_{k}(-1) - y_{s}(-1)\},\forall s\in \mathcal{S}$ and $Z_{s}(0) = R_s(0) + y_{s}(-1) - \sum_{k\in \mathcal{P}_{s}} x_{k}(-1), \forall s\in \mathcal{S}$. 
\item Each link $l$'s algorithm: At each time $t\in\{0,1,\ldots\}$, link $l$ does the following:
\begin{enumerate}
\item Receive the path rates $x_{k}(t)$ that use link $l$. Update $Q_l(t)$ via: 
\begin{align*}
Q_{l}(t+1) = \max\Big\{- \sum_{k\in \mathcal{D}_{l}} x_{k}(t) + c_{l}, Q_{l}(t) + \sum_{k\in \mathcal{D}_{l}} x_{k}(t) - c_{l}\Big\}. 
\end{align*}
\item The price of this link is given by $Y_{l}(t+1) = Q_{l}(t+1) + \sum_{k\in \mathcal{D}_{l}} x_{k}(t) - c_{l}$.
\item Communicate the link price $Y_{l}(t+1)$ to sources that use link $l$.
\end{enumerate}
\item Each source $s$'s algorithm: At each time $t\in\{0,1,\ldots\}$, source $s$ does the following:
\begin{enumerate}
\item Receive from the network the link prices $Y_{l}(t)$ for all links $l$ that are used by any path of source $s$.
\item Update the path rates $x_{k}, k\in \mathcal{P}_{s}$ by  
\begin{align*}
x_{k}(t) &= \argmin_{0 \leq x_{k}\leq x_{k}^{\max}} \Big\{ \big[\sum_{l\in \mathcal{E}_{k}}Y_{l}(t)  - Z_{s}(t)\big] x_{k} + \alpha (x_{k} -x_{k}(t-1))^{2}\Big\}\\
&= \Big[ x_{k}(t-1) - \frac{1}{2\alpha}\big(\sum_{l\in \mathcal{E}_{k}}Y_{l}(t)  - Z_{s}(t)\big)  \Big]_{0}^{x_{k}^{\max}},
\end{align*}
where $[z]^{b}_{a} = \min\{\max\{z, a\},b\}$.
\item Communicate the path rates $x_{k}(t), k\in \mathcal{P}_{s}$ to all links that are used by path $k$.
\item Update the source rate $y_{s}(t)$ by 
\begin{align*}
y_{s} (t)= \argmin_{0\leq y_{s} \leq y_{s}^{\max}} \Big\{-U_{s}(y_{s}) + Z_{s}(t) y_{s} + \alpha (y_{s} - y_{s}(t-1))^{2}\Big\},
\end{align*}
which usually has a closed form solution for differentiable utilities by taking derivatives.
\item Update virtual queue $R_{s}(t)$ and source price $Z_{s}(t)$ locally by  
\begin{align*}
R_{s}(t+1) &= \max\Big\{ -y_{s}(t) + \sum_{k\in \mathcal{P}_s} x_{k}(t), R_{s}(t) + y_{s}(t) -  \sum_{k\in \mathcal{P}_s} x_{k}(t)\Big\},\\
Z_{s}(t+1) &= R_{s}(t+1) + y_{s}(t) -  \sum_{k\in \mathcal{P}_s} x_{k}(t).
\end{align*}
\end{enumerate}
\end{itemize}
\end{algorithm}
By \cref{thm:overall-convergence}, if we choose $\alpha \geq \frac{1}{2} \big(S+K + \sum_{k\in \mathcal{K}} d_k\big)$, then for any $t\geq 1$, 
\begin{align}
\sum_{s=1}^{S} U_{s}(\overline{y}_{s}(t)) &\geq  \sum_{s=1}^{S} U_{s}(y_{s}^{\ast}) - O(1/t), \label{eq:net-utility} \\
\sum_{k\in \mathcal{D}_{l}} \overline{x}_{k}(t)  &\leq c_{l}  + O(1/t), \qquad \qquad\quad~~\forall l\in \mathcal{L},\\
\overline{y}_{s}(t) &\leq  \sum_{k\in \mathcal{P}_s} \overline{x}_{k}(t) + O(1/t),  \qquad \forall s\in \mathcal{S},  
\end{align} 
where $(\mathbf{x}^{\ast}, \mathbf{y}^{\ast})$ is an optimal solution of the problem \eqref{eq:network-obj}-\eqref{eq:network-source-box-cons}.  
Note that if \cref{alg:network-flow} has been run for a sufficiently long time and we are satisfied with the current performance, then we can fix $\mathbf{x} = \frac{1}{t}\sum_{\tau=0}^{t-1} \mathbf{x}(\tau)$ and $\mathbf{y}= \frac{1}{t}\sum_{\tau=0}^{t-1} \mathbf{y}(\tau)$ such that  the performance  is still within $O(1/t)$ sub-optimality for all future time.

\subsection{Decentralized Joint Flow and Power Control}

Since \cref{alg:dpp} allows for general nonlinear convex constraint functions, it can also be applied to solve the joint flow and power control problem.  In this case, the capacity of each link $l$ is not fixed but depends concavely on 
a power allocation variable $p_l$.  Assuming each link capacity is logarithmic in $p_l$ results in the following problem: 
\begin{align}
\text{maximize}~\quad & \sum_{s=1}^{S} U_{s}(y_{s}) - \sum_{l=1}^L V_l(p_l)\\
\text{subject to} \quad  &  \sum_{k\in \mathcal{D}_l} x_k \leq \log(1+p_l), \forall l\in \mathcal{L},\\
			 &  y_{s} \leq \sum_{k\in \mathcal{P}_s} x_{k}  , \forall s\in \mathcal{S}, \\
			 &  0\leq x_{k} \leq x_{k}^{\max}, \forall k\in\mathcal{K}, \\
			 & 0 \leq y_{s} \leq y_{s}^{\max}, \forall s\in \mathcal{S}, \\
			 & 0 \leq p_l \leq p_{l}^{\max}, \forall l\in \mathcal{L},
\end{align}
where $p_l$ is the power allocated at link $l$, $\log(1+p_l)$ is the corresponding link capacity as a function of $p_l$, and $V_l(p_l)$ is the associated power cost  (assumed to be a convex function of $p_l$).  A decentralized joint flow and power control algorithm for this problem can be similarly developed by applying \cref{alg:dpp}. 

\section{Numerical Results}
This section considers numerical experiments to verify the convergence rate results shown in this paper. 

\subsection{Decentralized Multiplath Flow Control}

Consider the simple multipath network flow problem described in \cref{fig:network-flow}. Assume each link has capacity $1$. Let $y_{1}, y_{2}$ and $y_{3}$ be the data rates of source $1, 2$ and $3$; $x_1, x_2, x_3, x_4, x_5, x_{6}$ and $x_7$ be the data rates of the paths indicated in the figure; and the network utility be maximizing $\log (y_{1}) + 2\log(y_{2}) + 2\log(y_{3})$.  The NUM problem can be formulated as follows:
\begin{align*}
\text{maximize}~\quad & \log(y_1)+ 2\log(y_2) + 2\log(y_3) \\
\text{subject to} \quad  &  \mathbf{R}\mathbf{x} \leq \mathbf{c},\\
& \mathbf{y} \leq \mathbf{T}\mathbf{x},\\
& 0\leq x_i \leq x_{i}^{\max}, i\in\{1,2,\ldots,4\}, \\
& 0 \leq y_i \leq y_{i}^{\max}, i\in\{1,2,3\}, 
\end{align*} 
where $\mathbf{R} = \left[\begin{array}{ccccccc} 1 & 0 &0 & 0 &0&0&0 \\ 0 & 1&0&0&0&0&0\\ 0&0&1&0&0&0&0 \\ 1 & 0 &1&0 & 0 &0&0 \\0&1&0&1&0&0&0\\ 0&0&0&0&1&0&0\\0 & 0& 0 &0&1&1&0 \\ 0&0& 0&0 &0&1 &0\\ 0&0&0&0&0&0&1\end{array}\right],\mathbf{c} = \left[\begin{array}{c}1\\1\\1\\1\\1\\1\\1\\1\\1\end{array}\right]$, $\mathbf{T} = \left[\begin{array}{ccccccc} 1 &1 & 0 &0 & 0 &0&0\\ 0&0 & 1&1&1&0&0 \\ 0&0&0&0&0&1&1\end{array}\right]$. The optimal value to this NUM problem is $f^\ast = 1.65687$.

To verify the convergence of \cref{alg:network-flow}, \cref{fig:num_convergence} shows the values of objective and constraint functions yielded by \cref{alg:network-flow} with $\alpha = \frac{1}{2}\big(K+S+\sum_{k\in \mathcal{K}}d_k\big)+1 = 10$  and $\mathbf{x}(-1)=\mathbf{0}$.  (By writing constraints $\mathbf{R}\mathbf{x}\leq \mathbf{c}$ and $\mathbf{y} \leq \mathbf{T}\mathbf{x}$ in the compact form $\mathbf{A}\mathbf{z} \leq \mathbf{b}$, it can be checked that $\beta =\sigma_{\max}(\mathbf{A}) = 2.4307$. If we choose a smaller $\alpha$, e.g., $\alpha = \frac{1}{2}\beta^2+1 = 3.9543$, then \cref{alg:network-flow} converges even faster. In this simulation, we choose a loose $\alpha$ whose value can be easily estimated from \cref{lm:network-beta-bound} without knowing the detailed network topology.) We also compare our algorithm with the dual subgradient algorithm (with primal averaging) with step size $0.01$ in \cite{Nedic09}. (Or equivalently, the DPP algorithm with $V=100$ in \cite{Neely05DCDIS,Neely14Arxiv_ConvergenceTime,YuNeely15CDC}.) Recall that the dual subgradient algorithm in \cite{Neely05DCDIS,Nedic09, Neely14Arxiv_ConvergenceTime,YuNeely15CDC} does not converge to an exact optimal solution but only converges to an approximate solution with an error level determined by the step size. In contrast, our algorithm can eventually converge to the exact optimality.  \cref{fig:num_convergence} shows that \cref{alg:network-flow} converges faster than the dual subgradient algorithm with primal averaging.
  
To verify the convergence rate of \cref{alg:network-flow},  \cref{fig:num_convergence_rate} plots $f(\overline{\mathbf{x}}(t)) - f^\ast$, all constraint values,  function $1/t$, and bounds from \cref{thm:overall-convergence} with both x-axis and y-axis in $\log_{10}$ scales. It can be observed that the curves of $ f(\overline{\mathbf{x}}(t)) - f^\ast$ and all the source rate constraint values are parallel to the curve of $1/t$ for large $t$. Note that all the link capacity constraints are satisfied early (i.e., negative), and hence are not drawn in $\log_{10}$ scales. \cref{fig:num_convergence_rate} verifies that the error of \cref{alg:network-flow} decays like $O(1/t)$ and suggests that it is actually $\Theta(1/t)$ for this multipath NUM problem.

\begin{figure}[htbp]
\centering
   \includegraphics[width=0.9\textwidth,height=0.6\textheight,keepaspectratio=true]{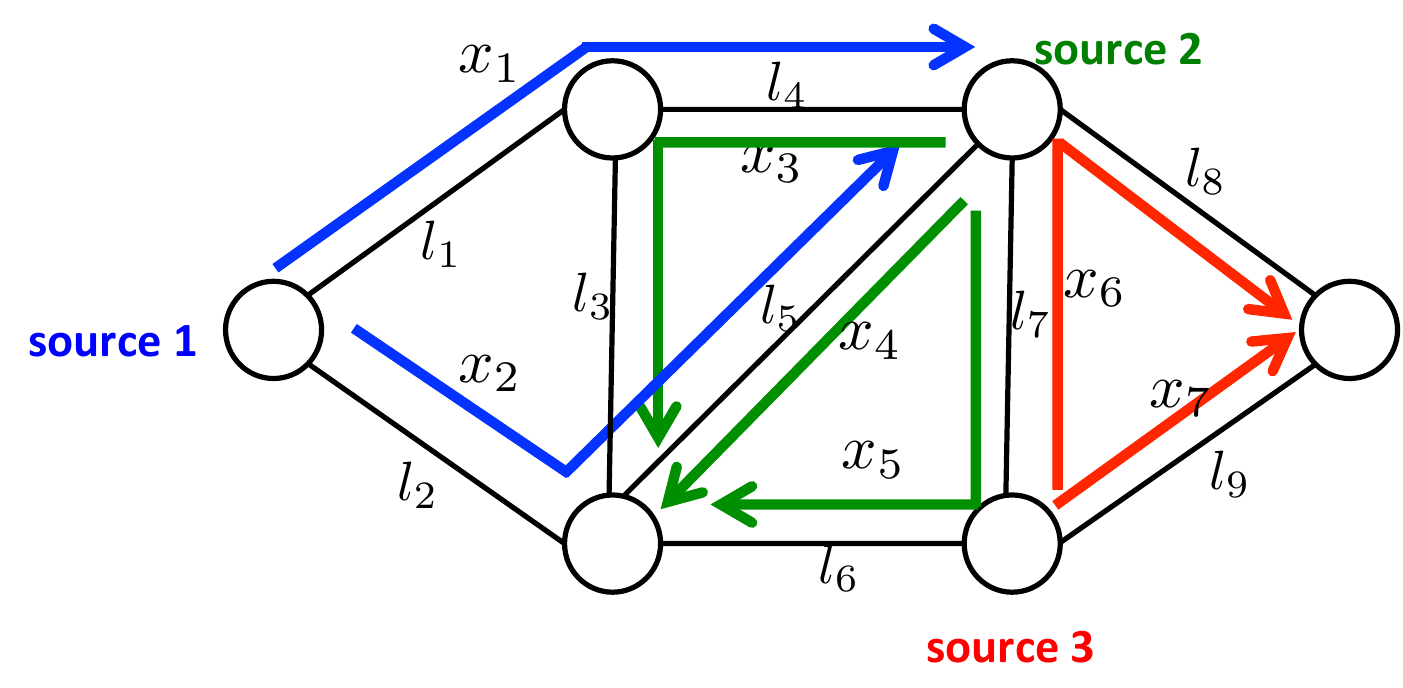} 
   \caption{A simple multipath NUM problem with $3$ sources and $7$ paths.}
   \label{fig:network-flow}
\end{figure}

\begin{figure}[htbp]
\centering
   \includegraphics[width=1\textwidth,height=0.9\textheight,keepaspectratio=true]{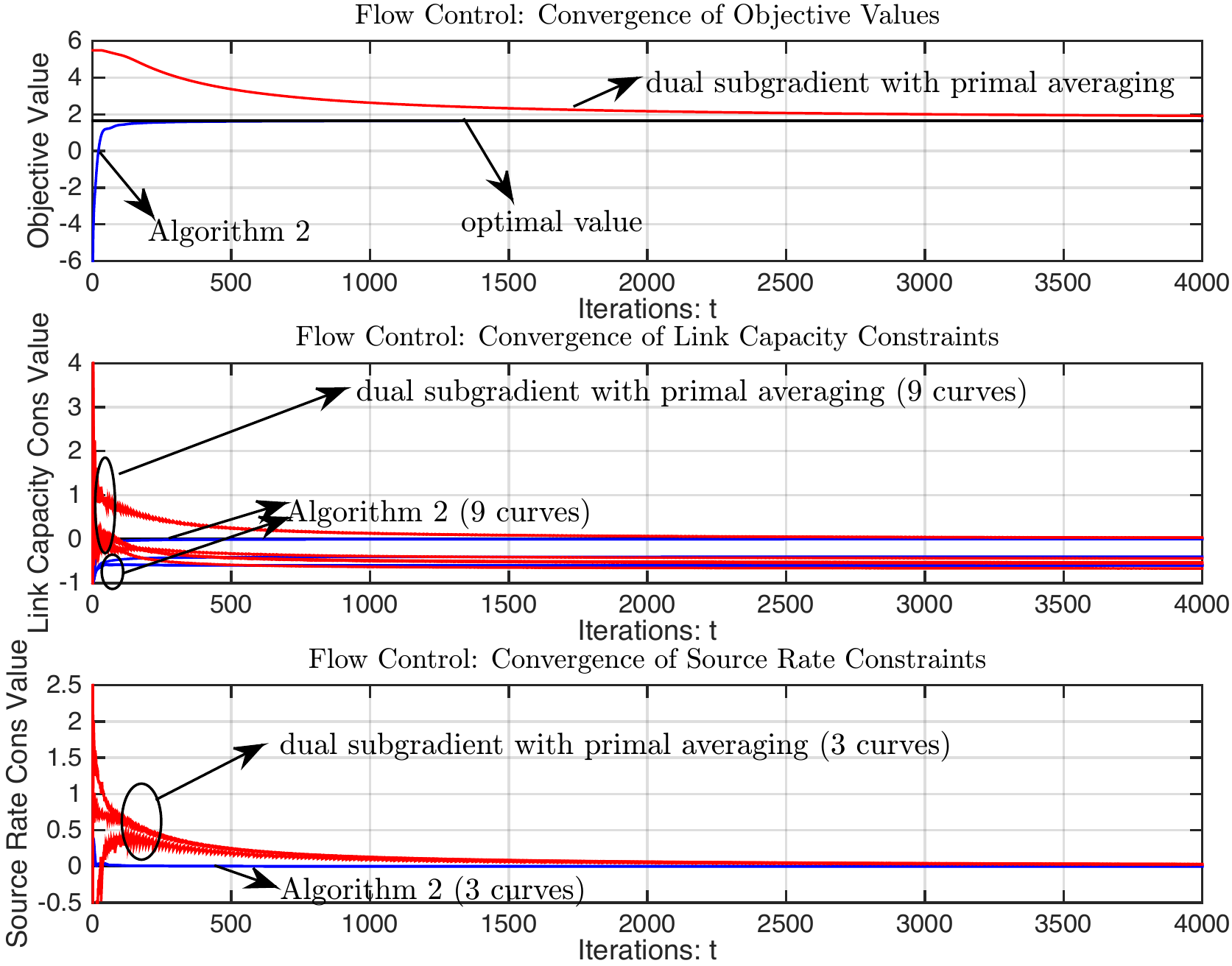} 
   \caption{The convergence of \cref{alg:network-flow} and the dual subgradient algorithm with primal averaging for a multipath flow control problem.}
   \label{fig:num_convergence}
\end{figure}

\begin{figure}[htbp]
\centering
   \includegraphics[width=1\textwidth,height=0.9\textheight,keepaspectratio=true]{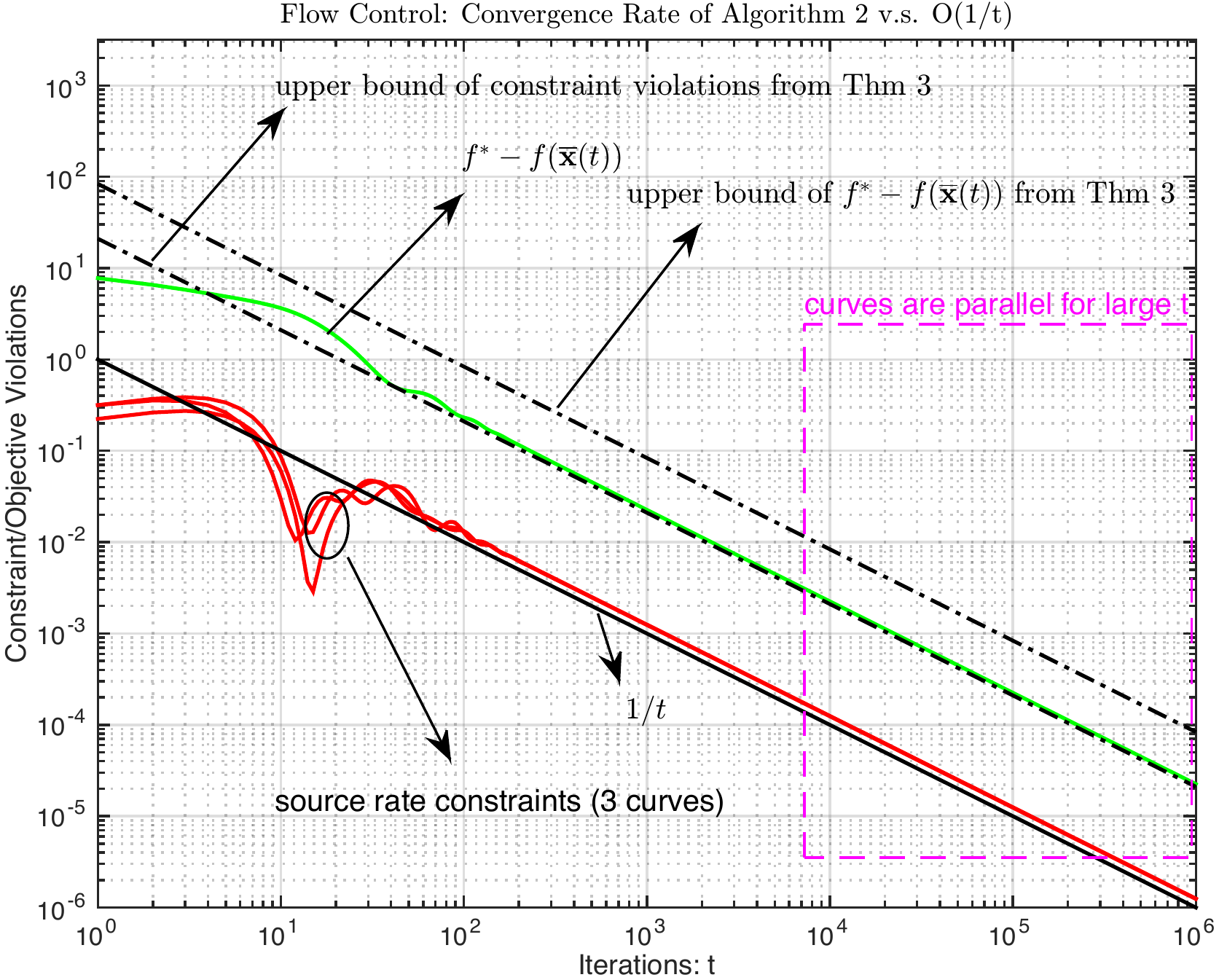} 
   \caption{The convergence rate of \cref{alg:network-flow} for a multipath flow control problem.}
   \label{fig:num_convergence_rate}
\end{figure}

\subsection{Decentralized Joint Flow and Power Control}

Consider the joint flow and power control over the same network described in \cref{fig:network-flow}. We assume the power cost of each link is given by $V_l(p_l) = 0.25 p_l$. The optimal value of this joint flow and power control problem is $f^\ast = -0.521318$.  

To verify the convergence of \cref{alg:dpp}, \cref{fig:power_convergence} shows the values of objective and constraint functions yielded by \cref{alg:dpp} with $\alpha = 10$, $\mathbf{x}(-1) =\mathbf{0}$, $\mathbf{y}(-1) = \mathbf{0}$ and $\mathbf{p}(-1) = \mathbf{0}$.  (In fact, by writing constraints $\mathbf{R}\mathbf{x}\leq \log(1+\mathbf{p})$ and $\mathbf{y} \leq \mathbf{T}\mathbf{x}$ in the compact form $\mathbf{g}(\mathbf{z})\leq \mathbf{0}$, it can be checked that $\beta = 2.5229$. If we choose a smaller $\alpha$, e.g., $\alpha = \frac{1}{2}\beta^2+1 = 4.1826$, then \cref{alg:dpp} converges even faster.) We also compare our algorithm with the dual subgradient algorithm (with primal averaging) with step size $0.01$ in \cite{Neely05DCDIS, Nedic09,Neely14Arxiv_ConvergenceTime, YuNeely15CDC}.  \cref{fig:power_convergence} shows that \cref{alg:dpp} converges faster than the dual subgradient algorithm with primal averaging.
  
To verify the convergence rate of \cref{alg:dpp},  \cref{fig:power_convergence_rate} plots $f(\overline{\mathbf{x}}(t)) - f^\ast$, all constraint values,  function $1/t$, and bounds from \cref{thm:overall-convergence} with both x-axis and y-axis in $\log_{10}$ scales. It can be observed that the curves of $f(\overline{\mathbf{x}}(t)) - f^\ast$ and all the source rate constraint values are parallel to the curve of $1/t$ for large $t$.  

\begin{figure}[htbp]
\centering
   \includegraphics[width=1\textwidth,height=0.9\textheight,keepaspectratio=true]{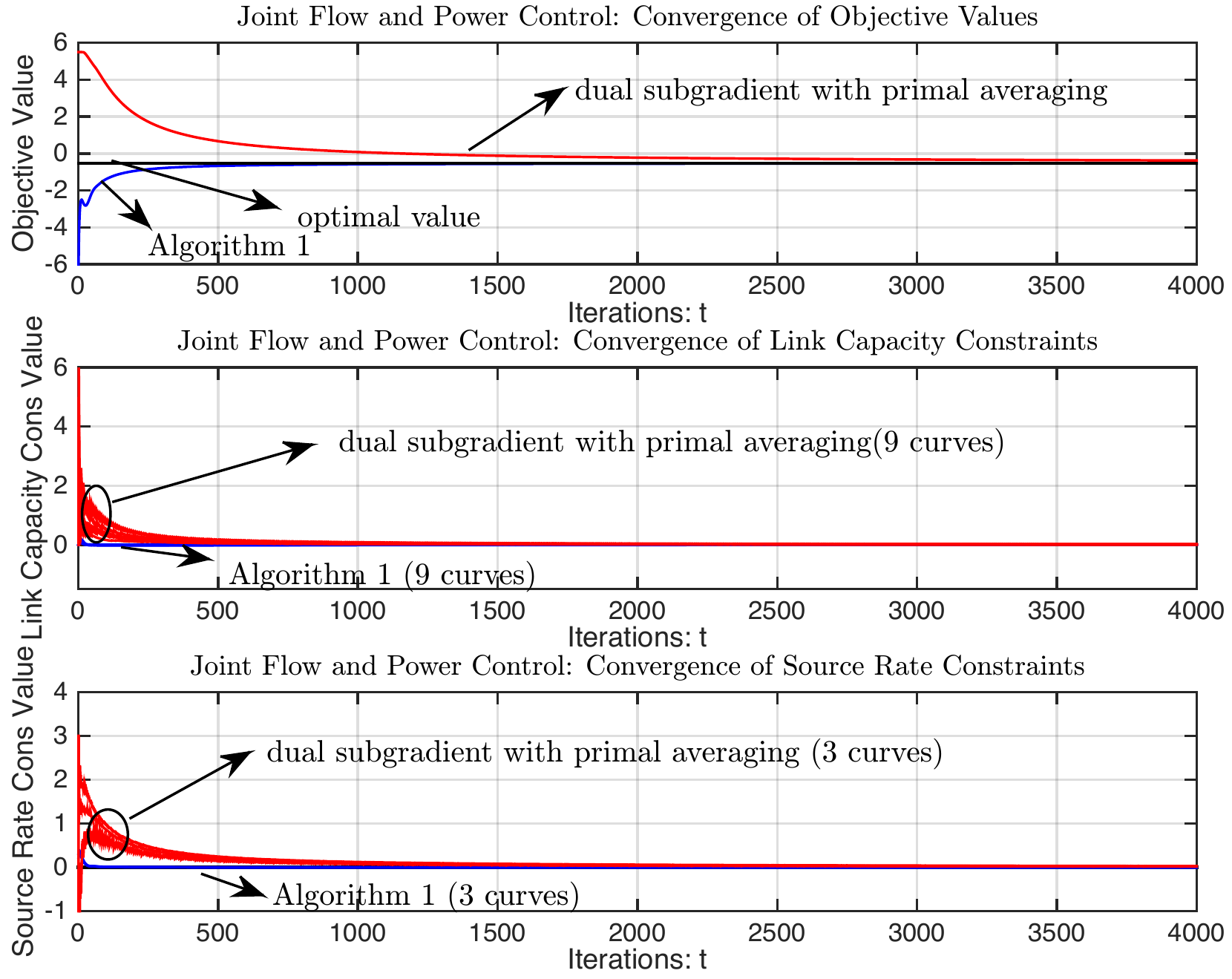} 
   \caption{The convergence of \cref{alg:dpp} and the dual subgradient algorithm with primal averaging for a multipath joint flow and power control problem.}
   \label{fig:power_convergence}
\end{figure}

\begin{figure}[htbp]
\centering
   \includegraphics[width=1\textwidth,height=0.9\textheight,keepaspectratio=true]{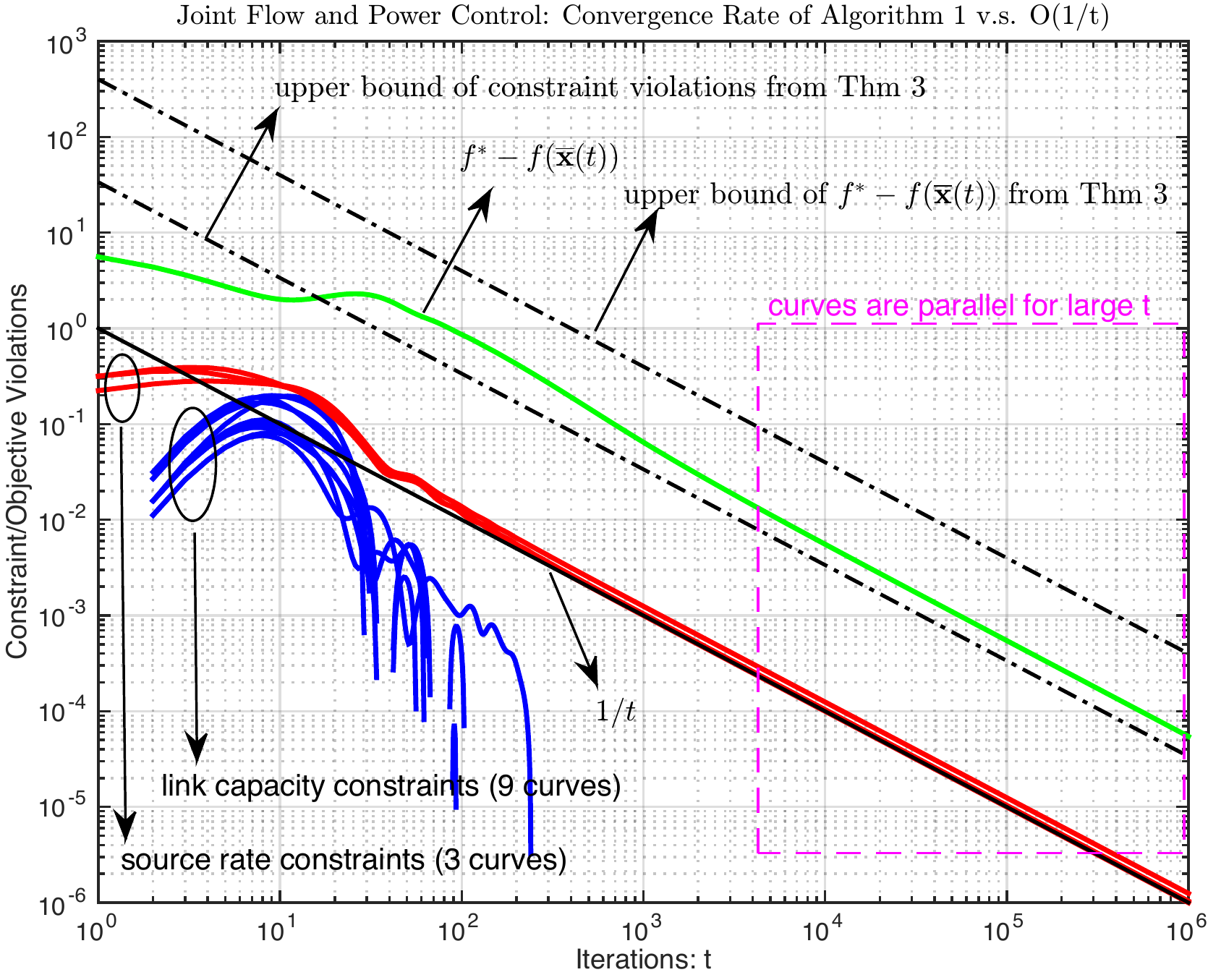} 
   \caption{The convergence rate of \cref{alg:dpp} for a multipath joint flow and power control problem.}
   \label{fig:power_convergence_rate}
\end{figure}

\subsection{Quadratic Programs}
Consider the following  quadratic program
\begin{align*}
\min~~ &  \mathbf{x}^T\mathbf{P}\mathbf{x} + \mathbf{c}^T \mathbf{x}\\
\text{s.t.} \quad  & \mathbf{x}^T \mathbf{Q} \mathbf{x} + \mathbf{d}^T \mathbf{x}\leq e\\
			 & \mathbf{x}^{\min} \leq \mathbf{x} \leq \mathbf{x}^{\max}
\end{align*}
where $\mathbf{P}$ and $\mathbf{Q}$ are positive semidefinite to ensure the convexity of the quadratic program.

We randomly generate a large scale example where $\mathbf{x}\in \mathbb{R}^{100}$, $\mathbf{P}\in \mathbb{R}^{100\times 100}$ is diagonal with entries from uniform $[0,4]$, $\mathbf{c}\in \mathbb{R}^{100}$ with entries from uniform $[-15,20]$, $\mathbf{Q}\in \mathbb{R}^{100\times 100}$ is diagonal with entries from uniform $[0,1]$, $\mathbf{d}\in \mathbb{R}^{100}$ with entries from uniform $[-1,1]$, $e$ is a scalar from uniform $[4,5]$, $\mathbf{x}^{\min} = \mathbf{0}$ and $\mathbf{x}^{\max} = \mathbf{1}$. Note that \cref{alg:dpp} and the dual subgradient algorithm (with primal averaging) can deal with general semidefinite positive matrices $\mathbf{P}$ and $\mathbf{Q}$. However, if $\mathbf{P}$ and $\mathbf{Q}$ are diagonal or block diagonal, then the primal update in both algorithms can be  decomposed into independent smaller problems and hence has extremely low complexity.

To verify the convergence of \cref{alg:dpp}, \cref{fig:qp_convergence} shows the values of objective and constraint functions yielded by \cref{alg:dpp} with $\alpha = \frac{1}{2}\beta^2+1$, where $\beta$ the is Lipschitz modulus of the constraint function, and $\mathbf{x}(-1) =\mathbf{x}^{\min}$.  We also compare our algorithm with the dual subgradient algorithm (with primal averaging) with step size $0.01$ in \cite{Neely05DCDIS, Nedic09,Neely14Arxiv_ConvergenceTime, YuNeely15CDC}.  \cref{fig:qp_convergence} shows that \cref{alg:dpp} converges faster than the dual subgradient algorithm with primal averaging.
  
To verify the convergence rate of \cref{alg:dpp},  \cref{fig:qp_convergence_rate} plots $f(\overline{\mathbf{x}}(t)) - f^\ast$,  function $1/t$, and the bound from \cref{thm:overall-convergence} with both x-axis and y-axis in $\log_{10}$ scales. It can be observed that the curves of $f(\overline{\mathbf{x}}(t)) - f^\ast$ and all the source rate constraint values are parallel to the curve of $1/t$ for large $t$.   Note that the constraint violation is not plotted since the constraint function is satisfied for all iterations as observed in \cref{fig:qp_convergence}.

\begin{figure}[htbp]
\centering
   \includegraphics[width=1\textwidth,height=0.9\textheight,keepaspectratio=true]{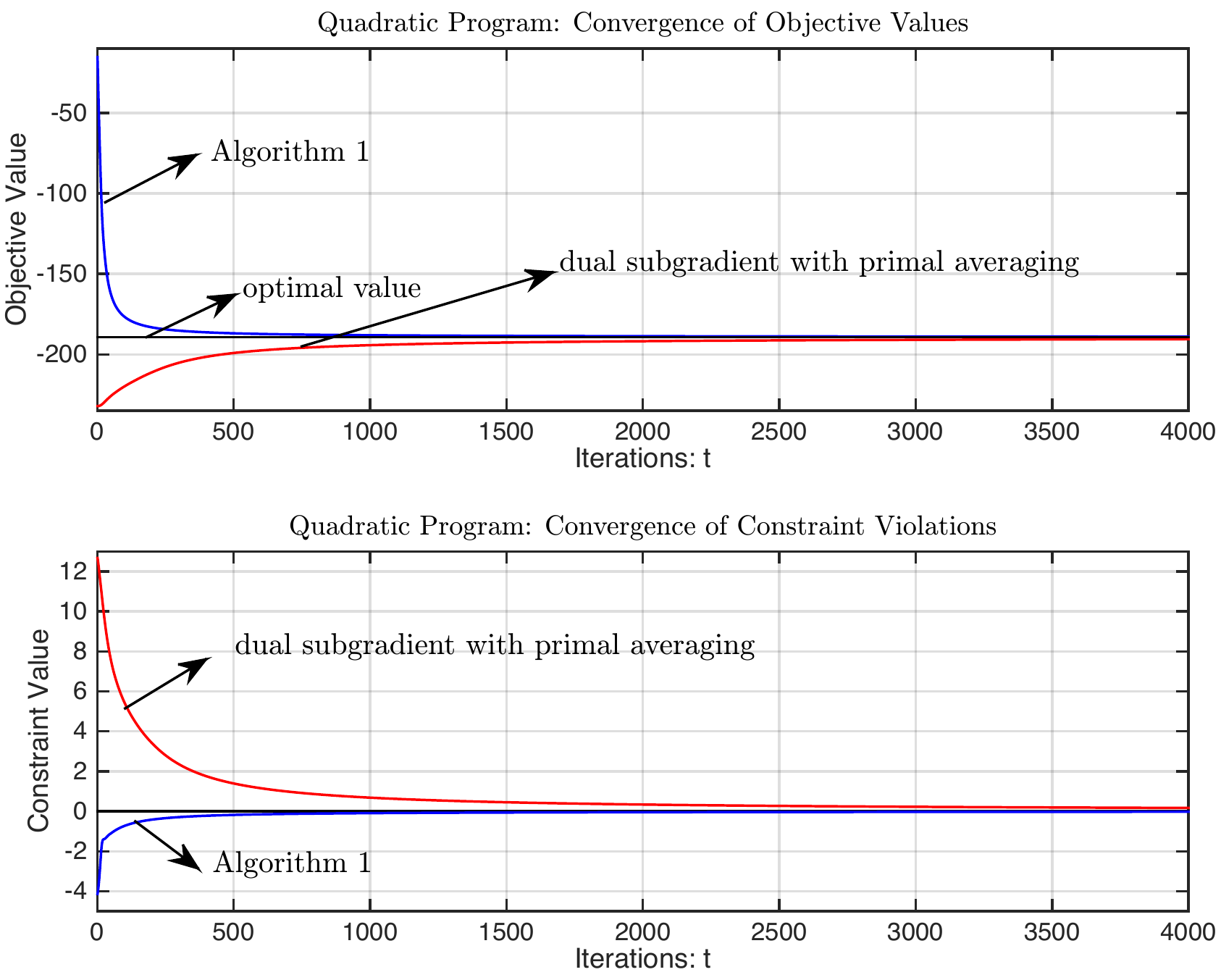} 
   \caption{The convergence of \cref{alg:dpp} and the dual subgradient algorithm with primal averaging for a quadratic program.}
   \label{fig:qp_convergence}
\end{figure}

\begin{figure}[htbp]
\centering
   \includegraphics[width=1\textwidth,height=0.9\textheight,keepaspectratio=true]{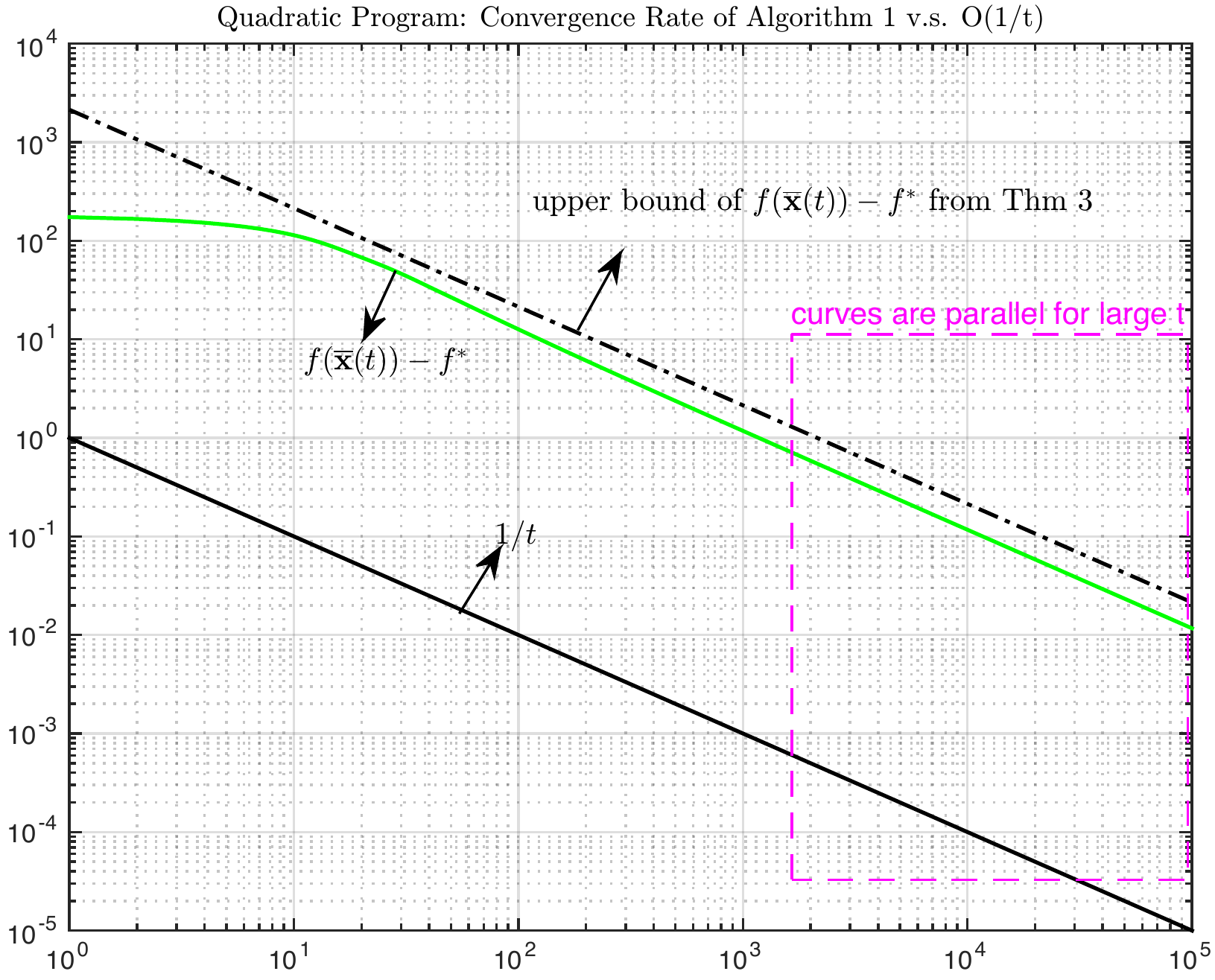} 
   \caption{The convergence rate of \cref{alg:dpp} for a quadratic program.}
   \label{fig:qp_convergence_rate}
\end{figure}

\section{Conclusions}
This paper proposes a novel but simple algorithm to solve convex programs with a possibly non-differentiable objective function and Lipschitz continuous constraint functions.  The new algorithm has a parallel implementation when the objective function and constraint functions are separable. The convergence rate of the proposed algorithm is shown to be $O(1/t)$. This is faster than the $O(1/\sqrt{t})$ convergence rate of the dual subgradient algorithm with primal averaging. The ADMM algorithm has the same $O(1/t)$ convergence rate but can only deal with linear equality constraint functions.  The new algorithm is further applied to solve multipath network flow control problems and yields a decentralized flow control algorithm which converges faster than existing dual subgradient or primal-dual subgradient based flow control algorithms.  The $O(1/t)$ convergence rate of the proposed algorithm is also verified by numerical experiments.

\bibliographystyle{siamplain}  
\bibliography{mybibfile}

\begin{thebibliography}{10}

\bibitem{book_NonlinearProgrammingTA}
{\sc M.~S. Bazaraa, H.~D. Sherali, and C.~M. Shetty}, {\em Nonlinear
  Programming: Theory and Algorithms}, Wiley-Interscience, 2006.

\bibitem{Beck14}
{\sc A.~Beck, A.~Nedic, A.~Ozdaglar, and M.~Teboulle}, {\em An ${O}(1/k)$
  gradient method for network resource allocation problems}, IEEE Transactions
  on Control of Network Systems, 1 (2014), pp.~64--73.

\bibitem{book_NonlinearProgramming_Bertsekas}
{\sc D.~P. Bertsekas}, {\em Nonlinear Programming}, Athena Scientific,
  second~ed., 1999.

\bibitem{Boyd11ADMMFoundatationTrends}
{\sc S.~Boyd, N.~Parikh, E.~Chu, B.~Peleato, and J.~Eckstein}, {\em Distributed
  optimization and statistical learning via the alternating direction method of
  multipliers}, Foundations and Trends in Machine Learning, 3 (2011),
  pp.~1--122.

\bibitem{book_ConvexOptimization}
{\sc S.~Boyd and L.~Vandenberghe}, {\em Convex Optimization}, Cambridge
  University Press, 2004.

\bibitem{He12SIAMNA}
{\sc B.~He and X.~Yuan}, {\em On the ${O}(1/n)$ convergence rate of the
  {D}ouglas-{R}achford alternating direction method}, SIAM Journal on Numerical
  Analysis, 50 (2012), pp.~700--709.

\bibitem{book_FundamentalConvexAnalysis}
{\sc J.-B. Hiriart-Urruty and C.~Lemar{\'e}chal}, {\em Fundamentals of Convex
  Analysis}, Springer, 2001.

\bibitem{Kelly98JORS}
{\sc F.~P. Kelly, A.~K. Maulloo, and D.~K. Tan}, {\em Rate control for
  communication networks: Shadow prices, proportional fairness and stability},
  Journal of the Operational Research Society,  (1998), pp.~237--252.

\bibitem{Lin15JORSC}
{\sc T.-Y. Lin, S.-Q. Ma, and S.-Z. Zhang}, {\em On the sublinear convergence
  rate of multi-block {ADMM}}, Journal of the Operations Research Society of
  China, 3 (2015), pp.~251--274.

\bibitem{Low00DualityModelTCP}
{\sc S.~H. Low}, {\em A duality model of {TCP} flow controls}, in Proceedings
  of ITC Specialist Seminar on IP Traffic Measurement, Modeling and Management,
  2000.

\bibitem{Low99TON}
{\sc S.~H. Low and D.~E. Lapsley}, {\em Optimization flow control---{I}: basic
  algorithm and convergence}, IEEE/ACM Transactions on Networking, 7 (1999),
  pp.~861--874.

\bibitem{Necoara14TAC}
{\sc I.~Necoara and V.~Nedelcu}, {\em Rate analysis of inexact dual first-order
  methods application to dual decomposition}, IEEE Transactions on Automatic
  Control, 59 (2014), pp.~1232--1243.

\bibitem{Nedic09}
{\sc A.~Nedi{\'c} and A.~Ozdaglar}, {\em Approximate primal solutions and rate
  analysis for dual subgradient methods}, SIAM Journal on Optimization, 19
  (2009), pp.~1757--1780.

\bibitem{Nedic09_PrimalDualSubgradient}
{\sc A.~Nedi{\'c} and A.~Ozdaglar}, {\em Subgradient methods for saddle-point
  problems}, Journal of Optimization Theory and Applications, 142 (2009),
  pp.~205--228.

\bibitem{Neely05DCDIS}
{\sc M.~J. Neely}, {\em Distributed and secure computation of convex programs
  over a network of connected processors}, in DCDIS Conference Guelph, 2005.

\bibitem{book_Neely10}
{\sc M.~J. Neely}, {\em Stochastic Network Optimization with Application to
  Communication and Queueing Systems}, Morgan \& Claypool Publishers, 2010.

\bibitem{Neely14Arxiv_ConvergenceTime}
{\sc M.~J. Neely}, {\em A simple convergence time analysis of
  drift-plus-penalty for stochastic optimization and convex programs},
  arXiv:1412.0791,  (2014).

\bibitem{book_ConvexOpt_Nesterov}
{\sc Y.~Nesterov}, {\em Introductory Lectures on Convex Optimization: A Basic
  Course}, Springer Science \& Business Media, 2004.

\bibitem{Parikh13ProximalAlgorithm}
{\sc N.~Parikh and S.~Boyd}, {\em Proximal algorithms}, Foundations and Trends
  in Optimization, 1 (2013), pp.~123--231.

\bibitem{Low03PE}
{\sc W.-H. Wang, M.~Palaniswami, and S.~H. Low}, {\em Optimal flow control and
  routing in multi-path networks}, Performance Evaluation, 52 (2003),
  pp.~119--132.

\bibitem{Wei13AsynchronousADMM}
{\sc E.~Wei and A.~Ozdaglar}, {\em On the ${O}(1/k)$ convergence of
  asynchronous distributed alternating direction method of multipliers}, in
  Proc. IEEE Global Conference on Signal and Information Processing, 2013.

\bibitem{Wei13TAC-1}
{\sc E.~Wei, A.~Ozdaglar, and A.~Jadbabaie}, {\em A distributed {N}ewton method
  for network utility maximization--{I}: algorithm}, IEEE Transactions on
  Automatic Control, 58 (2013), pp.~2162--2175.

\bibitem{YuNeely15CDC}
{\sc H.~Yu and M.~J. Neely}, {\em On the convergence time of the
  drift-plus-penalty algorithm for strongly convex programs}, in Proceedings of
  IEEE Conference on Decision and Control (CDC), 2015.

\bibitem{YuNeely16CDC}
{\sc H.~Yu and M.~J. Neely}, {\em A primal-dual type algorithm with the
  ${O}(1/t)$ convergence rate for large scale constrained convex programs}, in
  Proceedings of IEEE Conference on Decision and Control (CDC), 2016.

\end{thebibliography}

\end{document}